\documentclass[11pt]{article}
\usepackage{amsmath}
\usepackage{amssymb}
\usepackage{amsfonts}
\usepackage{hyperref}
\usepackage{mathrsfs}
\usepackage{cite}
\usepackage{cleveref}
\newtheorem{theorem}{Theorem}[section]

\newtheorem{lemma}{Lemma}[section]

\newtheorem{proposition}{Proposition}[section]

\numberwithin{equation}{section}

\newenvironment{proof}{{\noindent \bf Proof.}}{\hfill $\Box$}
\newenvironment{proof3.2}{{\noindent \bf Proof of Theorem 3.2.}}{\hfill $\Box$}
\newenvironment{proof4.2}{{\noindent \bf Proof of Theorem 4.2.}}{\hfill $\Box$}
\allowdisplaybreaks[4]

\begin{document}

\setlength{\baselineskip}{16pt}{\setlength\arraycolsep{2pt}}

\title{Continuous data assimilation for the three dimensional primitive equations with magnetic field}

\author{Yongqing Zhao$^1$,\ \ Wenjun Liu$^{1,2}$\footnote{Corresponding author. \ \   Email address: wjliu@nuist.edu.cn (W. J. Liu).}, \ \ Guangying Lv$^{1,2}$\ \ and   \ Yuepeng Wang$^{1,2}$  \\
    \medskip\\
    1. School of Mathematics and Statistics, Nanjing University of Information  \\
	Science and Technology, Nanjing 210044, China
	\medskip\\
	2.  Jiangsu International Joint Laboratory on System Modeling and Data Analysis, \\Center for Applied Mathematics of Jiangsu Province, Nanjing University \\ of Information Science and Technology, Nanjing 210044, China
}

\date{}
\maketitle

\begin{abstract}

In this paper, the problem of continuous data assimilation of three dimensional primitive equations with magnetic field in thin domain is studied. We establish the well-posedness of the assimilation system and prove that the $H^2$-strong solution of the assimilation system converges exponentially to the reference solution in the sense of $L^2$ as $t\rightarrow \infty$. We also study the sensitivity analysis of the assimilation system and prove that a sequence of solutions of the difference quotient equation converge to the unique solution of the formal sensitivity equation.
\end{abstract}

\noindent {\bf 2010 Mathematics Subject Classification:} 35Q35, 34D06, 37C50. \\
\noindent {\bf Keywords:}
Well-posedness,
data assimilation, sensitivity analysis.

\maketitle

\section{Introduction }

Continuous data assimilation is one of the classical data assimilation methods, which is to apply the time continuous and space discrete observation data to the assimilation process.
In 2014, based on the background of 2D Navier-Stokes equations, Azouani, Olson and Titi \cite{da2,da3} gave a new algorithm of continuous data assimilation inspired by ideas from control theory. In this new algorithm, they introduced a feedback control term into the original system.
\cite{da6, da9, da13, da17} also studied different systems separately by using finite parameter feedback control algorithms.
Let us consider the reference system
\begin{align}
&\frac{du}{dt}=F(u),\label{5.1.1}\\
&u(0)=u_0,\label{5.1.2}
\end{align}
Solutions of the system are called the reference solutions.
A main difficulty in obtaining the reference solution is that we largely do not know the initial data $u_0$. Under the assumption that there are no observation errors and the observation grid resolution is $h$, a linear interpolation operator $I_h$ and a relaxation (nudging) parameter $\beta$ were introduced into the system in \cite{da2,da3}. Thus, we obtain the corresponding assimilation system
\begin{align}
&\frac{d\tilde{u}}{dt}=F(\tilde{u})-\beta(I_h(u)-I_h(\tilde{u})),\label{5.1.3}\\
&\tilde{u}(0)=\tilde{u}_0,\label{5.1.4}
\end{align}
where $\tilde{u}_0$ is an arbitrary initial data. The solution of assimilation system \eqref{5.1.3}-\eqref{5.1.4} is called the assimilation solution.
The assimilation system \eqref{5.1.3}-\eqref{5.1.4} obviously eliminates the need for complete initial data and incorporates incoming data into simulations as well.
\cite{da2,da3} showed that the assimilation solution converges exponentially to the reference solution as time $t$ approaches infinity, under certain conditions on the spatial scale $h$ and the relaxation parameter $\beta$. The continuous data assimilation method introduced in \cite{da2,da3} is also called the data assimilation for simplicity. After the work in \cite{da2,da3}, the data assimilation method was also used to study the data assimilation problem for the different systems \cite{da1,da4,da10,da11,da18,da20,da21}.

Both the primitive equations and the Navier-Stokes equations are widely studied in various aspects. Cao and Titi \cite{pe1} achieved an important result on the global well-posedness of strong solutions for the three-dimensional primitive equations in a general cylindrical domain. Several works addressed the continuous data assimilation problem for the 2D Navier-Stokes equations \cite{da32,da25,da26,da12,da5,da8,da7,da30}. In 2021, Carlson and Larios \cite{da27,da29} performed parameter recovery for the 2D Navier-Stokes equations and further examined the sensitivity of their recovery algorithm by proving that a sequence of difference quotients converges to a unique solution of the sensitivity equations. Li and Titi \cite{pe4} established the small aspect ratio limit of the Navier–Stokes equations to the primitive equations in 2019. In \cite{da23,da31}, Korn and Pei investigated data assimilation for the primitive equations of ocean dynamics, respectively.

The magnetohydrodynamic (MHD) equations, which combine the Navier-Stokes equations of fluid dynamics and Maxwell’s equations of electromagnetism, have much richer structure than the Navier-Stokes equations. Duvaut and Lions \cite{pe2} established the existence and uniqueness results for weak and strong solutions of the 2D MHD equations. Biswas and Hudson \cite{da14} investigated the exponential convergence of continuous data assimilation for the 2D MHD equations in 2021. In \cite{pe3}, Du and Li analyzed the 3D incompressible MHD equations by using the scaling technique to derive the primitive equations with magnetic field (PEM) on thin 3D domains, and proved the global existence and uniqueness of strong solutions for this PEM without any small assumption on the initial data. In \cite{pe8}, they further proved that the global strong solutions of the scaled MHD equations strong converge to the global strong solutions of the PEM.



In this paper, we continue the  work of  Du and Li \cite{pe3,pe8} to further investigate the continuous data assimilation for three-dimensional PEM on the 3D  thin domains.
First, we obtain some uniform estimates to ensure the well-posedness of both the reference system and the assimilation system. Here, to obtain $L^4$-uniform estimates, we use the governing equation for Els\"asser variables to ensure that the nonlinear terms on the right-hand side of the energy inequality are absorbed by the left-hand side. For the assimilation system, we further estimate the assimilation terms using the properties of the linear interpolation operator $I_h$ and the incompressibility conditions.
Then, we apply the uniformly bounded estimates obtained previously to show that the solution of the assimilation system exponentially converges to that of the reference system.
Lastly, we study the sensitivity analysis problem for the assimilation system and establish the well-posedness of the formal sensitivity equations by proving that a sequence of solutions of the difference quotients equations converge to a unique solution of the formal sensitivity equations.

We organize the rest of this paper as follows. In Section 2, we give the mathematical formulation of primitive equations with magnetic data assimilation system. In Section 3, we first obtain the well-posedness of the PEM and its assimilation system, and then we get the convergence relation between the reference solution and the assimilation solution. In Section 4, we study the sensitivity analysis problem for the assimilation equations.

\section{Mathmatical Model}

We first consider the incompressible three-dimensional MHD equations in $\Omega_{\varepsilon}:=M\times(-\varepsilon,\varepsilon)$, where $\varepsilon>0$ is a small parameter, and $M=(0,L_1)\times(0,L_2)$ for two positive constants $L_1$ and $L_2$ of order $O(1)$ with respect to $\varepsilon$.
We notice the anisotropic MHD system
\begin{align*}
&\partial_t{u'}+(u'\cdot\nabla_3)u'+\nabla_3{p}-(b'\cdot\nabla_3)b'-\mu\Delta{u'}-\nu_1\partial_z^2{u'}=0, \\
&\partial_t{b'}+(u'\cdot\nabla_3)b'-(b'\cdot\nabla_3)u'-\kappa\Delta{b'}-\sigma_1\partial_z^2{b'}=0, \\
&\nabla_3\cdot{u'}=0,\ \nabla_3\cdot{b'}=0,
\end{align*}
where$$u'=(u,u_3),\ b'=(b,b_3),\ u=(u_1,u_2),\ b=(b_1,b_2),$$
and $u, b$ are the horizontal velocity and magnetic field, respectively, $u_3, b_3$ are vertical velocity and magnetic field. The pressure $p$ is scalar function. Throughout this paper, we set $\nabla=(\partial_x, \partial_y)$, $\Delta=\partial^2_{x}+\partial^2_{y}$, and $\nabla_3$ and $\Delta_3$ as their three-dimensional counterparts, respectively. Without loss of generality, we assume that $\nu_1=\nu\varepsilon^2,\ \sigma
_1=\sigma\varepsilon^2$, since $\mu,\kappa$ are of $O(1)$ order and $\nu_1,\sigma_1$ are of $O(\varepsilon^2)$ order, respectively, so that $\nu,\sigma$ are of $O(1)$ order as well. Next, we consider the above anisotropic viscosity and magnetic diffusivity scaling in the horizontal and vertical directions, so that the MHD equations can converge to PEM, as the aspect ratio $\varepsilon$ tends to zero.

We perform the scaling transformation on the MHD equations so that the resulting system is defined on a fixed domain that does not depend on $\varepsilon$, as in the case considered in \cite{pe3,pe4}.
To this end, we introduce the new unknowns
$$
u'_\varepsilon=(u_\varepsilon,u_{3,\varepsilon}),\ b'_\varepsilon=(b_\varepsilon,b_{3,\varepsilon}),\
p_\varepsilon(x,y,z,t)=p(x,y,\varepsilon z,t),
$$
$$
u_\varepsilon(x,y,z,t)=u(x,y,\varepsilon z,t)=(u_1(x,y,\varepsilon z,t),u_2(x,y,\varepsilon z,t)),
$$
$$
b_\varepsilon(x,y,z,t)=b(x,y,\varepsilon z,t)=(b_1(x,y,\varepsilon z,t),b_2(x,y,\varepsilon z,t)),
$$
$$
u_{3,\varepsilon}(x,y,z,t)=\frac{1}{\varepsilon}u_3(x,y,\varepsilon z,t),\ b_{3,\varepsilon}(x,y,z,t)=\frac{1}{\varepsilon}b_3(x,y,\varepsilon z,t),
$$
for any $(x,y,z)\in\Omega:=M\times(-1,1)$, and for any $t\in(0,\infty)$. Then, $u'_\varepsilon,\ b'_\varepsilon$ and $p_\varepsilon$ satisfy the scaled MHD equations (SMHD)
\begin{align*}
&\partial_t{u_\varepsilon}+(u'_\varepsilon\cdot\nabla_3)u_\varepsilon+\nabla{p_\varepsilon}-(b'_\varepsilon\cdot\nabla_3)b_\varepsilon-\mu\Delta{u_\varepsilon}-\nu\partial_z^2{u_\varepsilon}=0, \\
&\varepsilon^2\left[\partial_t{u_{3,\varepsilon}}+(u'_\varepsilon\cdot\nabla_3)u_{3,\varepsilon}-(b'_\varepsilon\cdot\nabla_3)b_{3,\varepsilon}-\mu\Delta{u_{3,\varepsilon}}-\nu\partial_z^2{u_{3,\varepsilon}}\right]=\partial_z{p_\varepsilon}, \\
&\partial_t{b_\varepsilon}+(u'_\varepsilon\cdot\nabla_3)b_\varepsilon-(b'_\varepsilon\cdot\nabla_3)u_\varepsilon-\kappa\Delta{b_\varepsilon}-\sigma\partial_z^2{b_\varepsilon}=0, \\
&\varepsilon^2\left[\partial_t{b_{3,\varepsilon}}+(u'_\varepsilon\cdot\nabla_3)b_{3,\varepsilon}-(b'_\varepsilon\cdot\nabla_3)u_{3,\varepsilon}-\mu\Delta{b_{3,\varepsilon}}-\nu\partial_z^2{b_{3,\varepsilon}}\right]=0, \\
&\nabla\cdot{u_\varepsilon}+\partial_z{u_{3,\varepsilon}}=0,\ \nabla\cdot{b_\varepsilon}+\partial_z{b_{3,\varepsilon}}=0,
\end{align*}
defined in the fixed domain $\Omega$.
In addition, we set the SMHD system with the  initial value conditions
\begin{align*}
(u_\varepsilon,u_{3,\varepsilon})\vert_{t=0}=(u_{\varepsilon,0},u_{3,\varepsilon,0}),\;
(b_\varepsilon,b_{3,\varepsilon})\vert_{t=0}=(b_{\varepsilon,0},b_{3,\varepsilon,0}),
\end{align*}
and periodic boundary conditions
\begin{align*}
u_\varepsilon,u_{3,\varepsilon},b_\varepsilon,b_{3,\varepsilon},p_\varepsilon \;{\rm{are}} \;{\rm{periodic}} \;{\rm{in}} \; x, y, z,
\end{align*}
where $(u_{\varepsilon,0},u_{3,\varepsilon,0})$ and $(b_{\varepsilon,0},b_{3,\varepsilon,0})$ are given.
We can notice that the periods in $x$ and $y$ are $L_1$ and $L_2$, respectively, while that in $z$ is $2$.
We suppose in addition that $u_\varepsilon,u_{3,\varepsilon},p_\varepsilon,b_\varepsilon,b_{3,\varepsilon}$ are even, odd, odd, even and odd with respect to $z$, respectively.
And we always assume that the initial horizontal velocity and magnetic field $u_0,b_0$ are periodic in $x,y,z$,
and are even in $z$.

By taking the limit as $\varepsilon\rightarrow 0$ in SMHD, we obtain the incompressible three-dimensional PEM system
\begin{align}
&\partial_t{u}+(u\cdot\nabla)u+u_3\partial_z{u}+\nabla{p}-(b\cdot\nabla)b-b_3\partial_z{b}+A_u{u}=0,\label{2.1} \\
&\partial_z{p}=0, \\
&\partial_t{b}+(u\cdot\nabla)b+u_3\partial_z{b}-(b\cdot\nabla)u-b_3\partial_z{u}+A_b{b}=0,\label{2.3} \\
&\nabla\cdot{u}+\partial_z{u_3}=0,\; \nabla\cdot{b}+\partial_z{b_3}=0,\label{2.4}
\end{align}
where
$$
u=(u_1,u_2),\;b=(b_1,b_2),\;A_u{u}=-\mu\Delta{u}-\nu\partial_z^2{u},\;A_b{b}=-\kappa\Delta{b}-\sigma\partial_z^2{b}.
$$
We recall that the periodic initial-boundary value problem for the SMHD equations, and it is clear that we should impose the same boundary conditions and symmetry conditions on the corresponding limiting PEM system. However, we only need to impose the initial conditions on the horizontal components of the velocity field and the magnetic field.
Since $u_{3,0}$ and $b_{3,0}$ are odd in $z$, we have $u_{3,0}(x,y,0)=b_{3,0}(x,y,0)=0$.
Then, $u_{3,0},b_{3,0}$ can be determined uniquely by the incompressibility conditions, namely,
\begin{align}
u_{3,0}(x,y,z)=-\int_{0}^z {\nabla \cdot u_0(x,y,\xi)} \,{\rm d}\xi,\\
b_{3,0}(x,y,z)=-\int_{0}^z {\nabla \cdot b_0(x,y,\xi)} \,{\rm d}\xi.
\end{align}
Similarly, $(u_3,b_3)$ can also be determined uniquely by $u,b$ via the incompressibility conditions as
\begin{align}
u_{3}(x,y,z,t)=-\int_{0}^z {\nabla \cdot u(x,y,\xi,t)} \,{\rm d}\xi,\label{2.9}\\
b_{3}(x,y,z,t)=-\int_{0}^z {\nabla \cdot b(x,y,\xi,t)} \,{\rm d}\xi.\label{2.10}
\end{align}

The reference system \eqref{2.1}-\eqref{2.4} corresponding to the assimilation system is
\begin{align}
&\partial_t{\tilde{u}}+(\tilde{u}\cdot\nabla)\tilde{u}+
{\tilde{u}}_3\partial_z{\tilde{u}}+\nabla{\tilde{p}}
-(\tilde{b}\cdot\nabla)\tilde{b}-{\tilde{b}}_3\partial_z{\tilde{b}}
+A_{u}{\tilde{u}}=-\beta_u(I_h(\tilde{u})-I_h(u)),\label{2.11}\\
&\partial_z{\tilde{p}}=0,\\
&\partial_t{\tilde{b}}+(\tilde{u}\cdot\nabla)\tilde{b}
+{\tilde{u}}_3\partial_z{\tilde{b}}-(\tilde{b}\cdot\nabla)\tilde{u}
-{\tilde{b}}_3\partial_z{\tilde{u}}+A_{b}{\tilde{b}}=-\beta_b(I_h(\tilde{b})-I_h(b)),\label{2.13}\\
&\nabla\cdot{\tilde{u}}+\partial_z{{\tilde{u}}_3}=0,\; \nabla\cdot{\tilde{b}}+\partial_z{{\tilde{b}}_3}=0,\label{2.14}
\end{align}
where the initial-boundary value is the same as the reference system.
We suppose that the linear interpolation operator satisfies that
\begin{align}
&{\vert\vert w-I_h(w) \vert\vert}_2 \leq C_h^{1}h{\vert\vert w \vert\vert}_{H^1},\; \forall w \in H^1,\label{2.15}\\
&{\vert\vert I_h(w) \vert\vert}_2 \leq C_h^{2}{\vert\vert w \vert\vert}_{2},\; \forall w \in L^2.\label{2.16}
\end{align}
For simplicity, we denote
\begin{align*}
B: H^1 \times H^1 \rightarrow {((H^2)^2)}{'},
\end{align*}
and
\begin{align*}
(B(u,\tilde{u}),\xi)=\int_\Omega {[(u \cdot \nabla)\tilde{u}+u_3\partial_z\tilde{u}]\cdot \xi} \,{\rm d}\Omega.
\end{align*}
A useful property of $B$ is $(B(u,\tilde{u}),\tilde{u})=0$, which was verified in \cite{pe1}.

The proofs in this paper rely on the following lemmas.

\begin{lemma}\label{lem2.1}
(\cite[Lemma 2.1]{le1})
The  inequalities
\begin{align}
\int_M {\left(\int_{-1}^1 {f(x,y,z)} \,{\rm d}z\right)\left(\int_{-1}^1 {g(x,y,z)h(x,y,z)} \,{\rm d}z\right)} \,{\rm d}M\nonumber\\
\leq C{\vert\vert f \vert\vert}_2^{\frac{1}{2}}\left({\vert\vert  f\vert\vert}_2^{\frac{1}{2}}+{\vert\vert \nabla f \vert\vert}_2^{\frac{1}{2}}\right)
{\vert\vert g \vert\vert}_2
{\vert\vert h \vert\vert}_2^{\frac{1}{2}}\left({\vert\vert h \vert\vert}_2^{\frac{1}{2}}+{\vert\vert \nabla h \vert\vert}_2^{\frac{1}{2}}\right),
\end{align}
and
\begin{align}
\int_M {\left(\int_{-1}^1 {f(x,y,z)} \,{\rm d}z\right)\left(\int_{-1}^1 {g(x,y,z)h(x,y,z)} \,{\rm d}z\right)} \,{\rm d}M\nonumber\\
\leq C{\vert\vert f \vert\vert}_2
{\vert\vert g \vert\vert}_2^{\frac{1}{2}}\left({\vert\vert  g \vert\vert}_2^{\frac{1}{2}}+{\vert\vert \nabla g \vert\vert}_2^{\frac{1}{2}}\right)
{\vert\vert h \vert\vert}_2^{\frac{1}{2}}\left({\vert\vert h \vert\vert}_2^{\frac{1}{2}}+{\vert\vert \nabla h \vert\vert}_2^{\frac{1}{2}}\right),
\end{align}
hold true, for any $f,g,h$ such that the right-hand sides make sense and are finite, where $C$ is a positiive constant depending only on $L_1$ and $L_2$.
\end{lemma}

\begin{lemma}\label{lem2.2}
(\cite[Lemma 1.1]{le2})
Let $g$, $h$ and $f$ be three non-negative locally integrable functions on $(t_0,\infty)$ such that
\begin{align*}
\frac{df}{dt} \leq gf+h, \;\forall t \geq t_0,
\end{align*}
and
\begin{align*}
\int_{t}^{t+1} {g} \,{\rm d}s \leq a_1,\; \int_t^{t+1} {h} \,{\rm d}s \leq a_2,\; \int_t^{t+1} {f} \,{\rm d}s \leq a_3,\;
\forall t \geq t_0,
\end{align*}
where $a_1,a_2,a_3$ are positive constants.
Then
\begin{align*}
f(t+1)\leq(a_2+a_3)e^{a_1},\; \forall t\geq t_0.
\end{align*}
\end{lemma}

\begin{lemma}\label{lem2.3}
(\cite[Lemma 2.2]{pe4})
Let $\varphi=(\varphi_1,\varphi_2,\varphi_3)$, $\phi$ and $\psi$ be periodic functions with basic domain $\Omega$. Suppose that $\varphi\in H^1(\Omega)$, with $\nabla_3\cdot{\varphi}=0$ in $\Omega$, $\int_{\Omega} {\varphi} \,{\rm d}\Omega=0$, and $\varphi_3\vert_{z=0}$, $\nabla_3\phi\in H^1(\Omega)$ and $\psi\in L^2(\Omega)$. Denote by $\varphi_H=(\varphi_1,\varphi_2)$ the horizontal components of the function $\varphi$. Then, we have the estimate
\begin{align*}
\left\vert\int_{\Omega} {(\varphi\cdot\nabla_3\phi)\psi} \,{\rm d}\Omega\right\vert\leq C{\vert\vert \nabla_3\varphi_H \vert\vert}_{2}^{\frac{1}{2}}{\vert\vert \Delta_3\varphi_H \vert\vert}_{2}^{\frac{1}{2}}
{\vert\vert \nabla_3\phi \vert\vert}_{2}^{\frac{1}{2}}{\vert\vert \Delta_3\phi \vert\vert}_{2}^{\frac{1}{2}}
{\vert\vert \psi \vert\vert}_{2}.
\end{align*}
where $C$ is a positive constant depending only on $L_1$ and $L_2$.
\end{lemma}

\begin{lemma}\label{lem2.4}
(\cite{le4})
Let $V, H, V'$ be three Hilbert spaces such that $V\subset H=H' \subset V'$, where $H'$ and $V'$ are the dual spaces of $H$ and $V$ respectively. Suppose $u\in L^2(0,T;V)$ and $\partial_t u\in L^2(0,T;V')$. Then $u$ is almost everywhere equal to a function continuous from $[0,T]$ into $H$.
\end{lemma}

\section{Well-posedness and exponential convergence}

In this section, we investigate the well-posedness of the PEM and its assimilation system, and consider the convergence relationship between the assimilation solution and the reference solution. To obtain these results, we first derive some uniform estimates. For simplicity, let's set $\mu=\kappa$ and $\nu=\sigma$ which can guarantee that the left-hand side of the energy inequality absorbs the nonlinear terms on the right-hand side.

To get these results, we first derive some consistent estimates. For simplicity, let's set $\mu=\kappa$and $\nu=\sigma$. This ensures that the energy inequality on the left absorbs the nonlinear term on the right.

\subsection{$L^2$-uniform estimates for $u, b, \tilde{u}, \tilde{b}$ }

Taking the inner product of equation \eqref{2.1} and \eqref{2.3} with $u$ and $b$ respectively, we have
\begin{align}
\frac{1}{2}\frac{d}{dt}{\vert\vert u \vert\vert}_{2}^2+\frac{1}{2}\frac{d}{dt}{\vert\vert b \vert\vert}_{2}^2
+\mu{\vert\vert \nabla u \vert\vert}_{2}^2+\nu{\vert\vert \partial_z{u} \vert\vert}_{2}^2
+\kappa{\vert\vert \nabla b \vert\vert}_{2}^2+\sigma{\vert\vert \partial_z{b} \vert\vert}_{2}^2=0.\label{3.1}
\end{align}
Thanks to $\mu=\kappa,\; \nu=\sigma$ and using the Poincar\'{e} inequality, we can infer that
\begin{align*}
\frac{d}{dt}\left({\vert\vert u \vert\vert}_{2}^2+{\vert\vert b \vert\vert}_{2}^2\right)
+2\mu C_1{\vert\vert u \vert\vert}_{2}^2+2\nu C_1{\vert\vert b \vert\vert}_{2}^2\leq0,
\end{align*}
so we can get
\begin{align*}
\frac{d}{dt}\left({\vert\vert u \vert\vert}_{2}^2+{\vert\vert b \vert\vert}_{2}^2\right)
+C\left({\vert\vert u \vert\vert}_{2}^2+{\vert\vert b \vert\vert}_{2}^2\right)\leq0,
\end{align*}
where $C=\min\{2\mu C_1, 2\nu C_1\}$.
Using the Gronwall inequality, we get
\begin{align*}
{\vert\vert u \vert\vert}_{2}^2+{\vert\vert b \vert\vert}_{2}^2 \leq
e^{-Ct}\left({\vert\vert u_0 \vert\vert}_{2}^2+{\vert\vert b_0 \vert\vert}_{2}^2\right),
\end{align*}
so $u$ and $b$ are $L^2$ uniformly bounded.
Integrating in time $0$ to $t$ for \eqref{3.1}, we can get that there exists $k_0>0$ such that
\begin{align}
{\vert\vert u \vert\vert}_{2}^2+{\vert\vert b \vert\vert}_{2}^2
+\int_0^{t} {\left({\vert\vert u \vert\vert}_{H^1}^2+{\vert\vert b \vert\vert}_{H^1}^2\right)} \,{\rm d}s \leq k_0,\label{3.2}
\end{align}
for $\forall t\in[0,\infty)$. For convenience, $k_0$ is used for boundedness throughout the rest of this paper.
Integrating in time $t$ to $t+1$ for \eqref{3.1}, we can get the uniform boundedness of the time average of
${\vert\vert u \vert\vert}_{H^1}^2$ and ${\vert\vert b \vert\vert}_{H^1}^2$,
i.e.,
\begin{align*}
\int_t^{t+1} {\left({\vert\vert u \vert\vert}_{H^1}^2+{\vert\vert b \vert\vert}_{H^1}^2\right)} \,{\rm d}s \leq k_0.
\end{align*}
Thus, we get
\begin{align}
{\vert\vert u \vert\vert}_{2}^2+{\vert\vert b \vert\vert}_{2}^2
+\int_t^{t+1} {\left({\vert\vert u \vert\vert}_{H^1}^2+{\vert\vert b \vert\vert}_{H^1}^2\right)} \,{\rm d}s \leq k_0.\label{3.3}
\end{align}

Similarly, taking the inner product of assimilation equation \eqref{2.11} and \eqref{2.13} with $\tilde{u}$ and $\tilde{b}$ respectively, we have
\begin{align*}
&\frac{1}{2}\frac{d}{dt}{\vert\vert \tilde{u} \vert\vert}_{2}^2+\frac{1}{2}\frac{d}{dt}{\vert\vert \tilde{b} \vert\vert}_{2}^2
+\mu{\vert\vert \nabla \tilde{u} \vert\vert}_{2}^2+\nu{\vert\vert \partial_z{\tilde{u}} \vert\vert}_{2}^2
+\kappa{\vert\vert  \nabla \tilde{b} \vert\vert}_{2}^2+\sigma{\vert\vert \partial_z{\tilde{b}} \vert\vert}_{2}^2\\
=&-\beta_u(I_h(\tilde{u})-I_h(u),\tilde{u})-\beta_b(I_h(\tilde{b})-I_h(b),\tilde{b}).
\end{align*}
Because of \eqref{2.15} and \eqref{2.16}, using the H\"older and Young inequalities, we can get
\begin{align*}
&-\beta_u(I_h(\tilde{u})-I_h(u),\tilde{u})-\beta_b(I_h(\tilde{b})-I_h(b),\tilde{b})\\
\leq& \beta_u{\vert\vert I_h(u) \vert\vert}_2{\vert\vert \tilde{u} \vert\vert}_2
-\beta_u(I_h(\tilde{u})-\tilde{u},\tilde{u})-\beta_u{\vert\vert \tilde{u} \vert\vert}^2_2\\
&+\beta_b{\vert\vert I_h(b) \vert\vert}_2{\vert\vert \tilde{b} \vert\vert}_2
-\beta_b(I_h(\tilde{b})-\tilde{b},\tilde{b})-\beta_b{\vert\vert \tilde{b} \vert\vert}^2_2\\
\leq& C\beta_u{\vert\vert u \vert\vert}_2{\vert\vert \tilde{u} \vert\vert}_2
+C_1\beta_u{\vert\vert \tilde{u} \vert\vert}_{H^1}{\vert\vert \tilde{u} \vert\vert}_2-\beta_u{\vert\vert \tilde{u} \vert\vert}^2_2\\
&+C\beta_b{\vert\vert b \vert\vert}_2{\vert\vert \tilde{b} \vert\vert}_2
+C_1\beta_b{\vert\vert \tilde{b} \vert\vert}_{H^1}{\vert\vert \tilde{b} \vert\vert}_2-\beta_b{\vert\vert \tilde{b} \vert\vert}^2_2\\
\leq& \frac{1}{2}C^2\beta_u{\vert\vert u \vert\vert}^2_{2}
+\frac{1}{2}C_1^2\beta_u{\vert\vert \tilde{u} \vert\vert}^2_{H^1}
+\frac{1}{2}C^2\beta_b{\vert\vert b \vert\vert}^2_{2}
+\frac{1}{2}C_1^2\beta_b{\vert\vert \tilde{b} \vert\vert}^2_{H^1},
\end{align*}
where $C=C_h^2,\; C_1=C_h^1h$.
Now choosing $h$ small enough so that
$$
C_1^2\beta_u \leq \min\{\mu,\nu\},\;
C_1^2\beta_b \leq \min\{\kappa,\sigma\},
$$
we get
\begin{align*}
&\frac{d}{dt}{\vert\vert \tilde{u} \vert\vert}_{2}^2+\frac{d}{dt}{\vert\vert \tilde{b} \vert\vert}_{2}^2
+\mu{\vert\vert \tilde{u} \vert\vert}_{2}^2+\nu{\vert\vert \partial_z{\tilde{u}} \vert\vert}_{2}^2
+\kappa{\vert\vert \tilde{b} \vert\vert}_{2}^2+\sigma{\vert\vert \partial_z{\tilde{b}} \vert\vert}_{2}^2\\
\leq& C^2\beta_u{\vert\vert u \vert\vert}^2_{2}+C^2\beta_b{\vert\vert b \vert\vert}^2_{2}.
\end{align*}
So by the Gronwall inequality and the estimate \eqref{3.2}, we get that $\tilde{u}$ and $\tilde{b}$ are $L^2$ uniformly bounded.
Similarly, by the estimate \eqref{3.3}, we can obtain
\begin{align}
{\vert\vert \tilde{u} \vert\vert}_{2}^2+{\vert\vert \tilde{b} \vert\vert}_{2}^2
+\int_t^{t+1} {\left({\vert\vert \tilde{u} \vert\vert}_{H^1}^2+{\vert\vert \tilde{b} \vert\vert}_{H^1}^2\right)} \,{\rm d}s \leq k_0.\label{3.4}
\end{align}

\subsection{$L^4$-uniform estimates for $u, b, \tilde{u}, \tilde{b}$}

Adding and subtracting the equations \eqref{2.1} and \eqref{2.3} respectively, we obtain
\begin{align*}
\partial_t&({u+b})+(u\cdot\nabla)(u+b)+u_3({\partial_z{u}+\partial_z{b}})+\nabla{p}-(b\cdot\nabla)(u+b)-b_3({\partial_z{u}+\partial_z{b}})\\
&-\mu\Delta{u}-\kappa\Delta{b}-\nu\partial_z^2{u}-\sigma\partial_z^2{b}=0,\\
\partial_t&({u-b})+(u\cdot\nabla)(u-b)+u_3({\partial_z{u}-\partial_z{b}})+\nabla{p}-(b\cdot\nabla)(u-b)-b_3\partial_z({u-b})\\
&-\mu\Delta{u}+\kappa\Delta{b}-\nu\partial_z^2{u}+\sigma\partial_z^2{b}=0.
\end{align*}
Letting $A=u+b,\;A^*=u-b$, we get
\begin{align}
\partial_t&A+(u\cdot\nabla)A+u_3\partial_z{A}+\nabla{p}-(b\cdot\nabla)A-b_3{\partial_z{A}}-\mu\Delta{u}-\kappa\Delta{b}-\nu\partial_z^2{u}-\sigma\partial_z^2{b}=0,\label{3.6}\\
\partial_t&A^*+(u\cdot\nabla)A^*+u_3{\partial_z{A^*}}+\nabla{p}-(b\cdot\nabla)A^*+b_3\partial_zA^*-\mu\Delta{u}+\kappa\Delta{b}-\nu\partial_z^2{u}+\sigma\partial_z^2{b}=0.\label{3.7}
\end{align}
For simplicity, we assume $\mu=\kappa,\;\nu=\sigma$.
Taking the inner product of equations \eqref{3.6} and \eqref{3.7} with ${\vert A \vert}^2A$ and ${\vert A^* \vert}^2A^*$ respectively, we obtain
\begin{align*}
&\frac{1}{4}\frac{d}{dt}{\vert\vert A \vert\vert}_4^4+\frac{1}{4}\frac{d}{dt}{\vert\vert A^* \vert\vert}_4^4\\
&+\int_M {\mu\nabla A(2\nabla{\vert A \vert}{\vert A \vert}A+\nabla A{\vert A \vert}^2)} \,{\rm d}M
+\int_z {\nu\partial_z A(2\partial_z{\vert A \vert}{\vert A \vert}A+\partial_z A{\vert A \vert}^2)} \,{\rm d}z\\
&+\int_M {\mu\nabla A^*(2\nabla{\vert A^* \vert}{\vert A^* \vert}A^*+\nabla A^*{\vert A^* \vert}^2)} \,{\rm d}M
+\int_z {\nu\partial_z A^*(2\partial_z{\vert A^* \vert}{\vert A^* \vert}A^*+\partial_z A^*{\vert A^* \vert}^2)} \,{\rm d}z\\
=&-\int_\Omega {\nabla p({\vert A \vert}^2A)} \,{\rm d}\Omega-\int_\Omega {\nabla p({\vert A^* \vert}^2A^*)} \,{\rm d}\Omega.
\end{align*}
Using Lemma \ref{lem2.1}, we infer that
\begin{align}
-\int_\Omega {\nabla p({\vert A \vert}^2A)} \,{\rm d}\Omega
\leq& \int_M {\left(\int_{-1}^1 {{\vert A \vert}^3} \,{\rm d}z\right)}\nabla p \,{\rm d}M\nonumber\\
\leq& C{\vert\vert \nabla p \vert\vert}_2{\vert\vert A \vert\vert}_4\left({\vert\vert A \vert\vert}_4
+{\vert\vert {\nabla {\vert A \vert}^2} \vert\vert}_2^{\frac{1}{2}}\right){\vert\vert  A \vert\vert}_2^{\frac{1}{2}}
\left({\vert\vert  A \vert\vert}_2^{\frac{1}{2}}+{\vert\vert  \nabla A \vert\vert}_2^{\frac{1}{2}}\right).
\end{align}
From \cite[Propsition 3.3]{pe3}, we know that
$$
{\vert\vert \nabla p \vert\vert}_2 \leq c{\vert\vert \vert A \vert \nabla A \vert\vert}_2.
$$
Because of the Poincar\'{e} and Young inequalities, we get
$$
\begin{aligned}
-\int_\Omega {\nabla p({\vert A \vert}^2A)} \,{\rm d}\Omega
\leq& C{\vert\vert \vert A \vert\nabla A \vert\vert}_2{\vert\vert A \vert\vert}_4\left({\vert\vert A \vert\vert}_4
+{\vert\vert {\nabla {\vert A \vert}^2} \vert\vert}_2^{\frac{1}{2}}\right){\vert\vert  A \vert\vert}_2^{\frac{1}{2}}
\left({\vert\vert  A \vert\vert}_2^{\frac{1}{2}}+{\vert\vert  \nabla A \vert\vert}_2^{\frac{1}{2}}\right)\\
\leq& C\left({\vert\vert \vert A \vert\nabla A \vert\vert}_2{\vert\vert A \vert\vert}_4^2
+{\vert\vert \vert A \vert\nabla A \vert\vert}_2^{\frac{3}{2}}{\vert\vert A \vert\vert}_4\right){\vert\vert  A \vert\vert}_2^{\frac{1}{2}}
{\vert\vert  \nabla A \vert\vert}_2^{\frac{1}{2}}\\
\leq& {\frac{1}{2}\mu}{\vert\vert \vert A \vert\nabla A \vert\vert}_2^{2}
+C_\mu\left({\vert\vert A \vert\vert}_2{\vert\vert \nabla A \vert\vert}_2+{\vert\vert A \vert\vert}_2^2{\vert\vert \nabla A \vert\vert}_2^2\right){\vert\vert A \vert\vert}_4^4.
\end{aligned}
$$
Similarly, we can get
\begin{align*}
-\int_\Omega {\nabla p({\vert A^* \vert}^2A^*)} \,{\rm d}\Omega
\leq& {\frac{1}{2}\mu}{\vert\vert \vert A^* \vert\nabla A^* \vert\vert}_2^{2}
+C_\mu\left({\vert\vert A^* \vert\vert}_2{\vert\vert \nabla A^* \vert\vert}_2+{\vert\vert A^* \vert\vert}_2^2{\vert\vert \nabla A^* \vert\vert}_2^2\right){\vert\vert A^* \vert\vert}_4^4.
\end{align*}
Combining all the above estimates, we obtain
\begin{align}
&\frac{d}{dt}\left({\vert\vert A \vert\vert}_4^4+{\vert\vert A^* \vert\vert}_4^4\right)
+\mu{\vert\vert \vert A \vert\nabla A \vert\vert}_2^2
+\mu{\vert\vert \vert A^* \vert\nabla A^* \vert\vert}_2^2+\nu{\vert\vert \vert A \vert\partial_z A \vert\vert}_2^2
+\nu{\vert\vert \vert A^* \vert\partial_z A^* \vert\vert}_2^2\nonumber\\
\leq& C\left({\vert\vert A \vert\vert}_2{\vert\vert \nabla A \vert\vert}_2+{\vert\vert A \vert\vert}_2^2{\vert\vert \nabla A \vert\vert}_2^2+{\vert\vert A^* \vert\vert}_2{\vert\vert \nabla A^* \vert\vert}_2+{\vert\vert A^* \vert\vert}_2^2{\vert\vert \nabla A^* \vert\vert}_2^2\right)\left({\vert\vert A \vert\vert}^4_4+{\vert\vert A^* \vert\vert}^4_4\right).
\end{align}
Considering the Gronwall inequality and estimate \eqref{3.2}, we can get
$A$ and $A^*$ are $L^4$ uniformly bounded.
Integrating in time $t$ to $t+1$ and by \eqref{3.3}, we can get
\begin{align}
{\vert\vert A \vert\vert}_4^4+{\vert\vert A^* \vert\vert}_4^4
+\int_t^{t+1} {\left({\vert\vert \vert A \vert\nabla A \vert\vert}_2^2
+{\vert\vert \vert A^* \vert\nabla A^* \vert\vert}_2^2+{\vert\vert \vert A \vert\partial_z A \vert\vert}_2^2
+{\vert\vert \vert A^* \vert\partial_z A^* \vert\vert}_2^2\right)} \,{\rm d}s \leq k_0.\label{3.9}
\end{align}

Similarly, adding and subtracting the equations \eqref{2.11} and \eqref{2.13} respectively, we obtain
\begin{align}
\partial_t&\tilde{A}+(\tilde{u}\cdot\nabla)\tilde{A}+\tilde{u}_3\partial_z{\tilde{A}}+\nabla{\tilde{p}}
-(\tilde{b}\cdot\nabla)\tilde{A}-\tilde{b}_3{\partial_z{\tilde{A}}}\nonumber\\
&-\mu\Delta{\tilde{A}}-\nu\partial_z^2{\tilde{A}}=-\beta_u(I_h(\tilde{u})-I_h(u))-\beta_b(I_h(\tilde{b})-I_h(b)),\label{3.11}\\
\partial_t&\tilde{A}^*+(\tilde{u}\cdot\nabla)\tilde{A}^*+\tilde{u}_3{\partial_z{\tilde{A}^*}}+\nabla{\tilde{p}}
-(\tilde{b}\cdot\nabla)\tilde{A}^*+\tilde{b}_3\partial_z\tilde{A}^*\nonumber\\
&-\mu\Delta{\tilde{A}^*}-\nu\partial_z^2{\tilde{A}^*}=-\beta_u(I_h(\tilde{u})-I_h(u))+\beta_b(I_h(\tilde{b})-I_h(b)).\label{3.12}
\end{align}
We also take the inner product of equations above with $ {\vert \tilde{A} \vert}^2\tilde{A}$ and ${\vert \tilde{A}^* \vert}^2\tilde{A}^*$, respectively, and get
\begin{align*}
&\frac{1}{4}\frac{d}{dt}{\vert\vert \tilde{A} \vert\vert}_4^4+\frac{1}{4}\frac{d}{dt}{\vert\vert \tilde{A}^* \vert\vert}_4^4+\mu{\vert\vert \vert \tilde{A} \vert\nabla \tilde{A} \vert\vert}_2^2
+\mu{\vert\vert \vert \tilde{A}^* \vert\nabla \tilde{A}^* \vert\vert}_2^2+\nu{\vert\vert \vert \tilde{A} \vert\partial_z \tilde{A} \vert\vert}_2^2
+\nu{\vert\vert \vert \tilde{A}^* \vert\partial_z \tilde{A}^* \vert\vert}_2^2\\
\leq&-\int_\Omega {\nabla \tilde{p}({\vert \tilde{A} \vert}^2\tilde{A})} \,{\rm d}\Omega-\int_\Omega {\nabla \tilde{p}({\vert \tilde{A}^* \vert}^2\tilde{A}^*)} \,{\rm d}\Omega\\
&+\int_\Omega {-\beta_u\left(I_h(\tilde{u})-I_h(u)\right)({\vert \tilde{A} \vert}^2\tilde{A})} \,{\rm d}\Omega
+\int_\Omega {-\beta_b\left(I_h(\tilde{b})-I_h(b)\right)({\vert \tilde{A} \vert}^2\tilde{A})} \,{\rm d}\Omega\\
&+\int_\Omega {-\beta_u\left(I_h(\tilde{u})-I_h(u)\right)({\vert \tilde{A}^* \vert}^2\tilde{A}^*)} \,{\rm d}\Omega
+\int_\Omega {\beta_b\left(I_h(\tilde{b})-I_h(b)\right)({\vert \tilde{A}^* \vert}^2\tilde{A}^*)} \,{\rm d}\Omega.
\end{align*}
In order to obtain ${\vert\vert \nabla \tilde{p} \vert\vert}_2$, we apply the operator $\int_{-1}^{1} {\nabla \cdot (\cdot)} \,{\rm d}z$ to equation \eqref{3.11} and obtain
\begin{align*}
&\int_{-1}^{1} {\nabla \cdot (\partial_t\tilde{A})} \,{\rm d}z+\int_{-1}^{1} {\nabla \cdot ((\tilde{u}\cdot\nabla)\tilde{A}+\tilde{u}_3\partial_z{\tilde{A}})} \,{\rm d}z
+\int_{-1}^{1} {\nabla \cdot (\nabla{\tilde{p}})} \,{\rm d}z\\
&-\int_{-1}^{1} {\nabla \cdot \left((\tilde{b}\cdot\nabla)A-\tilde{b}_3{\partial_z{\tilde{A}}}\right)} \,{\rm d}z
-\int_{-1}^{1} {\nabla \cdot (\mu\Delta{\tilde{A}}-\nu\partial_z^2{\tilde{A}})} \,{\rm d}z\\
=&\int_{-1}^{1} {\nabla \cdot \left[-\beta_u(I_h(\tilde{u})-I_h(u))-\beta_b\left(I_h(\tilde{b})-I_h(b)\right)\right]} \,{\rm d}z.
\end{align*}
Because of $u_3\vert_{z=0}=0,\;b_3\vert_{z=0}=0$, and from \eqref{2.9}-\eqref{2.10}, we get
$$
u_3=\int_{0}^z {\partial_z u_3} \,{\rm d}z=-\int_{0}^z {\nabla \cdot u} \,{\rm d}z,\;
b_3=\int_{0}^z {\partial_z b_3} \,{\rm d}z=-\int_{0}^z {\nabla \cdot b} \,{\rm d}z.
$$
By the H\"older inequality, we have
$$
{\vert\vert \partial_zu_3 \vert\vert}_2 \leq {\vert\vert \nabla u \vert\vert}_2,\;
{\vert\vert \partial_zb_3 \vert\vert}_2 \leq {\vert\vert \nabla b \vert\vert}_2.
$$
Similarly, we can get
$$
{\vert\vert \partial_z\tilde{u}_3 \vert\vert}_2 \leq {\vert\vert \nabla \tilde{u} \vert\vert}_2,\;
{\vert\vert \partial_z\tilde{b}_3 \vert\vert}_2 \leq {\vert\vert \nabla \tilde{b} \vert\vert}_2.
$$
So by the linear of $I_{h}$ and using the divergence free conditions, we can get the key estimates
\begin{align*}
&\left\vert\left\vert \int_{-1}^{1} {\nabla \cdot \left[ -\beta_u\left(I_h(\tilde{u})-I_h(u)\right)-\beta_b\left(I_h(\tilde{b})-I_h(b)\right)\right]} \,{\rm d}z \right\vert\right\vert_{2,M}\\
\leq& C\left\vert\left\vert \nabla \cdot \left[ -\beta_u(I_h(\tilde{u})-I_h(u))-\beta_b\left(I_h(\tilde{b})-I_h(b)\right)\right] \right\vert\right\vert_{2}\\
=& C\left\vert\left\vert \beta_u\left(I_h(\partial_z\tilde{u}_3)+I_h(\partial_zu_3)\right)+\beta_b\left(I_h(\partial_z\tilde{b}_3)+I_h(\partial_zb_3)\right) \right\vert\right\vert_{2}\\
\leq& C_\beta\left({\vert\vert \partial_z\tilde{u}_3 \vert \vert}_2+{\vert\vert \partial_zu_3 \vert\vert}_2+{\vert\vert \partial_z\tilde{b}_3 \vert\vert}_2+{\vert\vert \partial_zb_3 \vert\vert}_2\right)\\
\leq& C_\beta\left({\vert\vert \nabla \tilde{u} \vert \vert}_2+{\vert\vert \nabla u \vert\vert}_2+{\vert\vert \nabla \tilde{b} \vert\vert}_2+{\vert\vert \nabla b \vert\vert}_2\right).
\end{align*}
Thus, we have
$$
{\vert\vert \nabla \tilde{p} \vert\vert}_2 \leq c{\vert\vert  \tilde{A}\nabla \tilde{A} \vert\vert}_2+C_\beta\left({\vert\vert \nabla \tilde{u} \vert \vert}_2+{\vert\vert \nabla u \vert\vert}_2+{\vert\vert \nabla \tilde{b} \vert\vert}_2+{\vert\vert \nabla b \vert\vert}_2\right).
$$
Similarly, for \eqref{3.12}, we can get
$$
{\vert\vert \nabla \tilde{p} \vert\vert}_2 \leq c{\vert\vert  \tilde{A}^*\nabla \tilde{A}^* \vert\vert}_2+C_\beta\left({\vert\vert \nabla \tilde{u} \vert \vert}_2+{\vert\vert \nabla u \vert\vert}_2+{\vert\vert \nabla \tilde{b} \vert\vert}_2+{\vert\vert \nabla b \vert\vert}_2\right).
$$
So we can get
\begin{align*}
-\int_\Omega {\nabla \tilde{p}({\vert \tilde{A} \vert}^2\tilde{A})} \,{\rm d}\Omega
\leq& {\frac{1}{6}\mu}{\vert\vert \vert \tilde{A} \vert\nabla \tilde{A} \vert\vert}_2^{2}
+C_\mu^\beta\left({\vert\vert \tilde{A} \vert\vert}_2{\vert\vert \nabla \tilde{A} \vert\vert}_2+{\vert\vert \tilde{A} \vert\vert}_2^2{\vert\vert \nabla \tilde{A} \vert\vert}_2^2\right){\vert\vert \tilde{A} \vert\vert}_4^4,\\
-\int_\Omega {\nabla \tilde{p}({\vert \tilde{A}^* \vert}^2\tilde{A}^*)} \,{\rm d}\Omega
\leq& {\frac{1}{6}\mu}{\vert\vert \vert \tilde{A}^* \vert\nabla \tilde{A}^* \vert\vert}_2^{2}
+C_\mu^\beta\left({\vert\vert \tilde{A}^* \vert\vert}_2{\vert\vert \nabla \tilde{A}^* \vert\vert}_2+{\vert\vert \tilde{A}^* \vert\vert}_2^2{\vert\vert \nabla \tilde{A}^* \vert\vert}_2^2\right){\vert\vert \tilde{A}^* \vert\vert}_4^4,
\end{align*}
where
$$
C_\mu^\beta=C_\mu+C_\beta\left({\vert\vert \nabla \tilde{u} \vert \vert}_2+{\vert\vert \nabla u \vert\vert}_2+{\vert\vert \nabla \tilde{b} \vert\vert}_2+{\vert\vert \nabla b \vert\vert}_2\right).
$$
By the H\"older and Young inequalities, we get
\begin{align*}
\int_\Omega {-\beta_u\left(I_h(\tilde{u})-I_h(u)\right)({\vert \tilde{A}\vert}^2\tilde{A})} \,{\rm d}\Omega
\leq& C\beta_u\left( {\vert\vert I_h(\tilde{u}) \vert\vert}_2+{\vert\vert I_h(u) \vert\vert}_2 \right)
{\vert\vert \tilde{A} \vert\vert}_4 {\vert\vert \tilde{A}^2 \vert\vert}_4\\
\leq& C\beta_u C_h^2({\vert\vert \tilde{u} \vert\vert}_2+{\vert\vert u \vert\vert}_2)
{\vert\vert \tilde{A} \vert\vert}_4
{\vert\vert \tilde{A} \vert\vert}_4
\left({\vert\vert \tilde{A}\nabla\tilde{A} \vert\vert}_2+{\vert\vert \tilde{A}\partial_z\tilde{A} \vert\vert}_2\right)\\
\leq& C^* {\vert\vert \tilde{A} \vert\vert}_4^4+{\frac{1}{6}\mu}{\vert\vert \vert \tilde{A} \vert\nabla \tilde{A} \vert\vert}_2^{2}+{\frac{1}{4}\nu}{\vert\vert \vert \tilde{A} \vert\nabla \tilde{A} \vert\vert}_2^{2}.
\end{align*}
Similarly, we can get
\begin{align*}
\int_\Omega {-\beta_b\left(I_h(\tilde{b})-I_h(b)\right)({\vert \tilde{A}\vert}^2\tilde{A})} \,{\rm d}\Omega
\leq& C^* {\vert\vert \tilde{A} \vert\vert}_4^4+{\frac{1}{6}\mu}{\vert\vert \vert \tilde{A} \vert\nabla \tilde{A} \vert\vert}_2^{2}+{\frac{1}{4}\nu}{\vert\vert \vert \tilde{A} \vert\nabla \tilde{A} \vert\vert}_2^{2},\\
\int_\Omega {-\beta_u\left(I_h(\tilde{u})-I_h(u)\right)({\vert \tilde{A}^*\vert}^2\tilde{A}^*)} \,{\rm d}\Omega
\leq& C^* {\vert\vert \tilde{A}^* \vert\vert}_4^4+{\frac{1}{6}\mu}{\vert\vert \vert \tilde{A}^* \vert\nabla \tilde{A}^* \vert\vert}_2^{2}+{\frac{1}{4}\nu}{\vert\vert \vert \tilde{A}^* \vert\nabla \tilde{A}^* \vert\vert}_2^{2},\\
\int_\Omega {\beta_b\left(I_h(\tilde{b})-I_h(b)\right)({\vert \tilde{A}^*\vert}^2\tilde{A}^*)} \,{\rm d}\Omega
\leq& C^* {\vert\vert \tilde{A}^* \vert\vert}_4^4+{\frac{1}{6}\mu}{\vert\vert \vert \tilde{A}^* \vert\nabla \tilde{A}^* \vert\vert}_2^{2}+{\frac{1}{4}\nu}{\vert\vert \vert \tilde{A}^* \vert\nabla \tilde{A}^* \vert\vert}_2^{2}.
\end{align*}
Thanks to the above estimates, using the Gronwall inequality and the estimates \eqref{3.3}-\eqref{3.4}, we have
\begin{align}
{\vert\vert \tilde{A} \vert\vert}_4^4+{\vert\vert \tilde{A}^* \vert\vert}_4^4
+\int_t^{t+1} {\left({\vert\vert \vert \tilde{A} \vert\nabla \tilde{A} \vert\vert}_2^2
+{\vert\vert \vert \tilde{A}^* \vert\nabla \tilde{A}^* \vert\vert}_2^2+{\vert\vert \vert \tilde{A} \vert\partial_z \tilde{A} \vert\vert}_2^2
+{\vert\vert \vert \tilde{A}^* \vert\partial_z \tilde{A}^* \vert\vert}_2^2\right)} \,{\rm d}s \leq k_0.
\end{align}

\subsection{$L^2$-uniform estimates for $\partial_z u, \partial_z b, \partial_z \tilde{u}, \partial_z \tilde{b}$}

In order to obtain the $H^1$-uniform bounds on $u$ and $b$, we first need to get $L^2$-uniform estimates for $\partial_z u, \partial_z b$.
Taking the inner product of equation \eqref{2.1} and \eqref{2.3} with $-\partial_z^2 u$ and $-\partial_z^2 b$ respectively, we can get
\begin{align*}
&\frac{1}{2}\frac{d}{dt}\left({\vert\vert \partial_z u \vert\vert}_2^2+{\vert\vert \partial_z b \vert\vert}_2^2\right)
+\mu\left({\vert\vert \nabla\partial_z u \vert\vert}_2^2+{\vert\vert \nabla\partial_z b \vert\vert}_2^2\right)
+\nu\left({\vert\vert \partial_z^2 u \vert\vert}_2^2+{\vert\vert \partial_z^2 b \vert\vert}_2^2\right)\\
=& \int_\Omega {[(u\cdot\nabla)u+u_3\partial_z{u}]\partial_z^2 u} \,{\rm d}\Omega
+\int_\Omega {\left[(u\cdot\nabla)b+u_3\partial_z{b}\right]\partial_z^2 b} \,{\rm d}\Omega\\
&+\int_\Omega {\left[(b\cdot\nabla)b+b_3\partial_z{b}\right]\partial_z^2 u} \,{\rm d}\Omega
+\int_\Omega {\left[(b\cdot\nabla)u+b_3\partial_z{u}\right]\partial_z^2 b} \,{\rm d}\Omega.
\end{align*}
Using the method in \cite[Proposition 3.4]{pe3}, we can get
\begin{align*}
&\frac{d}{dt}\left({\vert\vert \partial_z u \vert\vert}_2^2+{\vert\vert \partial_z b \vert\vert}_2^2\right)
+\mu\left({\vert\vert \nabla\partial_z u \vert\vert}_2^2+{\vert\vert \nabla\partial_z b \vert\vert}_2^2\right)
+\nu\left({\vert\vert \partial_z^2 u \vert\vert}_2^2+{\vert\vert \partial_z^2 b \vert\vert}_2^2\right)\\
\leq& C\left({\vert\vert u \vert\vert}_4^8+{\vert\vert u \vert\vert}_4^4+{\vert\vert b \vert\vert}_4^4\right)
\left({\vert\vert \partial_z u \vert\vert}_2^2+{\vert\vert \partial_z b \vert\vert}_2^2\right).
\end{align*}
Considering the Gronwall inequality and \eqref{3.9}, we get that
$\partial_z u$ and $\partial_z b$ are $L^2$ uniformly bounded.
Integrating in time $t$ to $t+1$, we can get the uniform boundedness of the time average of
${\vert\vert \partial_z u \vert\vert}_{H^1}^2$ and ${\vert\vert \partial _z b \vert\vert}_{H^1}^2$.
Thus
\begin{align}
{\vert\vert \partial_z u \vert\vert}_2^2+{\vert\vert \partial_z b \vert\vert}_2^2
+\int_t^{t+1} {\left({\vert\vert \partial_z u \vert\vert}_{H^1}^2+{\vert\vert \partial_z b \vert\vert}_{H^1}^2\right)} \,{\rm d}s \leq k_0.
\end{align}

Similarly, in order to obtain the $H^1$-uniform bounds on $\tilde{u}$ and $\tilde{b}$, we first need to get $L^2$-uniform estimates for $\partial_z \tilde{u}, \partial_z \tilde{b}$.
Taking the inner product of equation \eqref{2.11} and \eqref{2.13} with $-\partial_z^2 \tilde{u}$ and $-\partial_z^2 \tilde{b}$ respectively, by Poincar\'{e}'s inequality, we get the key estimate
\begin{align*}
&\int_\Omega {\beta_u\left(I_h(\tilde{u})-I_h(u)\right)(-\partial_z^2 \tilde{u})} \,{\rm d}\Omega\\
\leq& C\beta_u\left( {\vert\vert I_h(\tilde{u}) \vert\vert}_2+{\vert\vert I_h(u) \vert\vert}_2 \right)
{\vert\vert \partial_z^2 \tilde{u} \vert\vert}_2 \\
\leq& C\beta_u^2\left( {\vert\vert I_h(\tilde{u}) \vert\vert}_2^2+{\vert\vert I_h(u) \vert\vert}_2^2 \right)
+\frac{1}{2}\nu{\vert\vert \partial_z^2 \tilde{u} \vert\vert}_2^2,
\end{align*}
and
\begin{align*}
&\int_\Omega {\beta_b\left(I_h(\tilde{b})-I_h(b)\right)(-\partial_z^2 \tilde{b})} \,{\rm d}\Omega\\
\leq& C\beta_b\left( {\vert\vert I_h(\tilde{b}) \vert\vert}_2+{\vert\vert I_h(b) \vert\vert}_2 \right)
{\vert\vert \partial_z^2 \tilde{b} \vert\vert}_2 \\
\leq& C\beta_b^2\left( {\vert\vert I_h(\tilde{b}) \vert\vert}_2^2+{\vert\vert I_h(b) \vert\vert}_2^2 \right)
+\frac{1}{2}\nu{\vert\vert \partial_z^2 \tilde{b} \vert\vert}_2^2.
\end{align*}
So we can get
\begin{align*}
&\frac{d}{dt}\left({\vert\vert \partial_z \tilde{u} \vert\vert}_2^2+{\vert\vert \partial_z \tilde{b} \vert\vert}_2^2\right)
+\mu\left({\vert\vert \nabla\partial_z \tilde{u} \vert\vert}_2^2+{\vert\vert \nabla\partial_z \tilde{b} \vert\vert}_2^2\right)
+\nu\left({\vert\vert \partial_z^2 \tilde{u} \vert\vert}_2^2+{\vert\vert \partial_z^2 \tilde{b} \vert\vert}_2^2\right)\\
\leq& C\left({\vert\vert \tilde{u} \vert\vert}_4^8+{\vert\vert \tilde{u} \vert\vert}_4^4+{\vert\vert \tilde{b} \vert\vert}_4^4\right)
\left({\vert\vert \partial_z \tilde{u} \vert\vert}_2^2+{\vert\vert \partial_z \tilde{b} \vert\vert}_2^2\right).
\end{align*}
Using the Gronwall inequality, we get that
$\partial_z \tilde{u}$ and $\partial_z \tilde{b}$ are $L^2$ uniformly bounded.
Integrating in time $t$ to $t+1$, we can get the uniform boundedness of the time average of
${\vert\vert \partial_z \tilde{u} \vert\vert}_{H^1}^2$ and ${\vert\vert \partial _z \tilde{b} \vert\vert}_{H^1}^2$.
Thus
\begin{align}
{\vert\vert \partial_z \tilde{u} \vert\vert}_2^2+{\vert\vert \partial_z \tilde{b} \vert\vert}_2^2
+\int_t^{t+1} {\left({\vert\vert \partial_z \tilde{u} \vert\vert}_{H^1}^2+{\vert\vert \partial_z \tilde{b} \vert\vert}_{H^1}^2\right)} \,{\rm d}s \leq k_0.
\end{align}

\subsection{$H^1$-uniform estimates for $u,b,\tilde{u},\tilde{b}$}

Taking the inner product of equation \eqref{2.1} and \eqref{2.3} with $-\Delta u$ and $-\Delta b$ respectively, we have
\begin{align*}
&\frac{1}{2}\frac{d}{dt}{\vert\vert \nabla u \vert\vert}_2^2+\frac{1}{2}\frac{d}{dt}{\vert\vert \nabla b \vert\vert}_2^2
+\mu\left({\vert\vert \Delta{u} \vert\vert}_2^2+{\vert\vert \Delta{b} \vert\vert}_2^2\right)
+\nu\left({\vert\vert \nabla\partial_z u \vert\vert}_2^2+{\vert\vert \nabla\partial_z b \vert\vert}_2^2\right)\\
\leq& \int_\Omega {[(u\cdot\nabla)u+u_3\partial_z{u}]\Delta u} \,{\rm d}\Omega +\int_\Omega {\left[(u\cdot\nabla)b+u_3\partial_z{b}\right]\Delta b} \,{\rm d}\Omega\\
&+\int_\Omega {\left[(b\cdot\nabla)b+b_3\partial_z{b}\right]\Delta u} \,{\rm d}\Omega +\int_\Omega {\left[(b\cdot\nabla)u+b_3\partial_z{u}\right]\Delta b} \,{\rm d}\Omega\\
=&:M_1+M_2+M_3+M_4,
\end{align*}
where
\begin{align*}
M_1 &= \int_\Omega {[(u\cdot\nabla)u+u_3\partial_z{u}]\Delta u} \,{\rm d}\Omega\\
&=\int_\Omega {(u\cdot\nabla)u\Delta u} \,{\rm d}\Omega
+\int_\Omega {u_3\partial_z{u}\Delta u} \,{\rm d}\Omega\\
&=:M_{11}+M_{12}.
\end{align*}
Using Lemma \ref{lem2.1}, the Young and Poincar\'{e} inequalities, we infer that
\begin{align*}
M_{11} =& \int_\Omega {(u\cdot\nabla)u\Delta u} \,{\rm d}\Omega\\
\leq& \int_M {\left(\int_{-1}^1 {(\vert u \vert+\vert \partial_z u \vert)} \,{\rm d}z\right)
\left(\int_{-1}^1 {\left(\vert \nabla u \vert+\vert \Delta u \vert\right)} \,{\rm d}z\right)} \,{\rm d}M\\
\leq& C\left[{\vert\vert u \vert\vert}_2^{\frac{1}{2}}\left({\vert\vert  u \vert\vert}_2+{\vert\vert \nabla u \vert\vert}_2\right)^{\frac{1}{2}}
+{\vert\vert \partial_z u \vert\vert}_2^{\frac{1}{2}}({\vert\vert \partial_z u \vert\vert}_2+{\vert\vert \nabla \partial_z  u \vert\vert}_2)^{\frac{1}{2}}\right]\\
&\cdot{\vert\vert \nabla u \vert\vert}_2^{\frac{1}{2}}
({\vert\vert \nabla u \vert\vert}_2+{\vert\vert \nabla^2 u \vert\vert}_2)^{\frac{1}{2}}
{\vert\vert \Delta u \vert\vert}_2\\
\leq& {\frac{1}{12}}\mu{\vert\vert \Delta u \vert\vert}_2^2+{\frac{1}{8}}\nu{\vert\vert \nabla\partial_z u \vert\vert}_2^2
+C\left({\vert\vert u \vert\vert}_2^2{\vert\vert \nabla u \vert\vert}_2^2+{\vert\vert \partial_z u \vert\vert}_2^2\right){\vert\vert \nabla u \vert\vert}_2^2.
\end{align*}
Because of $u_3\vert_{z=0}=0$, and from \eqref{2.9}, we get
$$
u_3=\int_{0}^z {\partial_z u_3} \,{\rm d}z=-\int_{0}^z {\nabla \cdot u} \,{\rm d}z.
$$
So we can obtain
\begin{align*}
M_{12} &= \int_\Omega {u_3\partial_z{u}\Delta u} \,{\rm d}\Omega\\
&\leq \int_M {\left(\int_{-1}^1 {\vert \nabla\cdot u \vert} \,{\rm d}z\right)
\left(\int_{-1}^1 {{\vert \partial_z u \vert}{\vert \Delta u \vert}} \,{\rm d}z\right)} \,{\rm d}M\\
&\leq C{\vert\vert u \vert\vert}_2^{\frac{1}{2}}\left({\vert\vert \nabla u \vert\vert}_2+{\vert\vert \nabla^2 u \vert\vert}_2\right)^{\frac{1}{2}}{\vert\vert \partial_z u \vert\vert}_2^{\frac{1}{2}}\left({\vert\vert \partial_z u \vert\vert}_2+{\vert\vert \nabla\partial_z u \vert\vert}_2\right)^{\frac{1}{2}}{\vert\vert \Delta u \vert\vert}_2\\
&\leq {\frac{1}{12}\mu}{\vert\vert \Delta u \vert\vert}_2^2+{\frac{1}{8}}\nu{\vert\vert \nabla\partial_z u \vert\vert}_2^2
+C{\vert\vert \partial_z u \vert\vert}_2^2{\vert\vert \nabla u \vert\vert}_2^2.
\end{align*}
Therefore, we have
\begin{align*}
M_{1} \leq& {\frac{1}{6}}\mu{\vert\vert \Delta u \vert\vert}_2^2+{\frac{1}{4}}\nu{\vert\vert \nabla\partial_z u \vert\vert}_2^2
+C\left({\vert\vert u \vert\vert}_2^2{\vert\vert \nabla u \vert\vert}_2^2+2{\vert\vert \partial_z u \vert\vert}_2^2\right){\vert\vert \nabla u \vert\vert}_2^2.
\end{align*}
Similarly, we can get
\begin{align*}
M_{2} \leq& {\frac{1}{6}}\mu{\vert\vert \Delta b \vert\vert}_2^2+{\frac{1}{12}}\mu{\vert\vert \Delta u \vert\vert}_2^2
+{\frac{1}{8}}\nu{\vert\vert \nabla\partial_z u \vert\vert}_2^2+{\frac{1}{8}}\nu{\vert\vert \nabla\partial_z b \vert\vert}_2^2\\
&+C\left({\vert\vert u \vert\vert}_2^2{\vert\vert \nabla u \vert\vert}_2^2
+{\vert\vert \partial_z u \vert\vert}_2^2\right){\vert\vert \nabla b \vert\vert}_2^2
+C{\vert\vert \partial_z b \vert\vert}_2^2{\vert\vert \nabla u \vert\vert}_2^2,\\
M_{3} \leq& {\frac{1}{6}}\mu{\vert\vert \Delta b \vert\vert}_2^2+{\frac{1}{6}}\mu{\vert\vert \Delta u \vert\vert}_2^2
+{\frac{1}{4}}\nu{\vert\vert \nabla\partial_z b \vert\vert}_2^2\\
&+C\left({\vert\vert b \vert\vert}_2^2{\vert\vert \nabla b \vert\vert}_2^2+2{\vert\vert \partial_z b \vert\vert}_2^2\right){\vert\vert \nabla b \vert\vert}_2^2,\\
M_{4} \leq& {\frac{1}{6}}\mu{\vert\vert \Delta b \vert\vert}_2^2+{\frac{1}{12}}\mu{\vert\vert \Delta u \vert\vert}_2^2
+{\frac{1}{8}}\nu{\vert\vert \nabla\partial_z u \vert\vert}_2^2+{\frac{1}{8}}\nu{\vert\vert \nabla\partial_z b \vert\vert}_2^2\\
&+C\left({\vert\vert b \vert\vert}_2^2{\vert\vert \nabla b \vert\vert}_2^2
+{\vert\vert \partial_z b \vert\vert}_2^2\right){\vert\vert \nabla u \vert\vert}_2^2
+C{\vert\vert \partial_z u \vert\vert}_2^2{\vert\vert \nabla b \vert\vert}_2^2.
\end{align*}
Thus
\begin{align*}
&\frac{d}{dt}\left({\vert\vert \nabla u \vert\vert}_2^2+{\vert\vert \nabla b \vert\vert}_2^2\right)
+\mu\left({\vert\vert \Delta{u} \vert\vert}_2^2+{\vert\vert \Delta{b} \vert\vert}_2^2\right)
+\nu\left({\vert\vert \nabla\partial_z u \vert\vert}_2^2+{\vert\vert \nabla\partial_z b \vert\vert}_2^2\right)\\
\leq& C\left({\vert\vert u \vert\vert}_2^2{\vert\vert \nabla u \vert\vert}_2^2+{\vert\vert \partial_z u \vert\vert}_2^2+{\vert\vert b \vert\vert}_2^2{\vert\vert \nabla b \vert\vert}_2^2+{\vert\vert \partial_z b \vert\vert}_2^2\right)
\left({\vert\vert \nabla u \vert\vert}_2^2+{\vert\vert \nabla b \vert\vert}_2^2\right).
\end{align*}
Considering Gronwall's inequality, we get that
$u$ and $b$ are $H^1$ uniformly bounded.
Integrating in time $t$ to $t+1$, we can get the uniform boundedness of the time average of
${\vert\vert u \vert\vert}_{H^2}^2$ and ${\vert\vert b \vert\vert}_{H^2}^2$.
Thus
\begin{align}
{\vert\vert u \vert\vert}_{H^1}^2+{\vert\vert b \vert\vert}_{H^1}^2
+\int_t^{t+1} {\left({\vert\vert u \vert\vert}_{H^2}^2+{\vert\vert b \vert\vert}_{H^2}^2\right)} \,{\rm d}s \leq k_0.
\end{align}

Similarly, taking the inner product of equation \eqref{2.11} and \eqref{2.13} with $-\Delta \tilde{u}$ and $-\Delta \tilde{b}$ respectively, we have
\begin{align*}
&\frac{1}{2}\frac{d}{dt}\left({\vert\vert \nabla \tilde{u} \vert\vert}_2^2+{\vert\vert \nabla \tilde{b} \vert\vert}_2^2\right)
+\mu\left({\vert\vert \Delta{\tilde{u}} \vert\vert}_2^2+{\vert\vert \Delta{\tilde{b}} \vert\vert}_2^2\right)
+\nu\left({\vert\vert \nabla\partial_z \tilde{u} \vert\vert}_2^2+{\vert\vert \nabla\partial_z \tilde{b} \vert\vert}_2^2\right)\\
\leq& \int_\Omega {\left[(\tilde{u}\cdot\nabla)\tilde{u}+\tilde{u}_3\partial_z{\tilde{u}}\right]\Delta \tilde{u}} \,{\rm d}\Omega +\int_\Omega {\left[(\tilde{u}\cdot\nabla)\tilde{b}+\tilde{u}_3\partial_z{\tilde{b}}\right]\Delta \tilde{b}} \,{\rm d}\Omega\\
&+\int_\Omega {\left[(\tilde{b}\cdot\nabla)\tilde{b}+\tilde{b}_3\partial_z{\tilde{b}}\right]\Delta \tilde{u}} \,{\rm d}\Omega +\int_\Omega {\left[(\tilde{b}\cdot\nabla)\tilde{u}+\tilde{b}_3\partial_z{\tilde{u}}\right]\Delta \tilde{b}} \,{\rm d}\Omega\\
&-\int_\Omega {-\beta_u\left(I_h(\tilde{u})-I_h(u)\right)(-\Delta \tilde{u})} \,{\rm d}\Omega-\int_\Omega {-\beta_b\left(I_h(\tilde{b})-I_h(b)\right)(-\Delta \tilde{b})} \,{\rm d}\Omega.
\end{align*}
We can get the key estimates
\begin{align*}
\int_\Omega {\beta_u\left(I_h(\tilde{u})-I_h(u)\right)(-\Delta \tilde{u})} \,{\rm d}\Omega
\leq& C\beta_u\left( {\vert\vert I_h(\tilde{u}) \vert\vert}_2+{\vert\vert I_h(u) \vert\vert}_2 \right)
{\vert\vert \Delta \tilde{u} \vert\vert}_2 \\
\leq& C\beta_u^2\left( {\vert\vert I_h(\tilde{u}) \vert\vert}_2^2+{\vert\vert I_h(u) \vert\vert}_2^2 \right)
+\frac{1}{4}\mu{\vert\vert \Delta \tilde{u} \vert\vert}_2^2,
\end{align*}
and
\begin{align*}
\int_\Omega {\beta_b\left(I_h(\tilde{b})-I_h(b)\right)(-\Delta \tilde{b})} \,{\rm d}\Omega
\leq& C\beta_b\left( {\vert\vert I_h(\tilde{b}) \vert\vert}_2+{\vert\vert I_h(b) \vert\vert}_2 \right)
{\vert\vert \Delta \tilde{b} \vert\vert}_2 \\
\leq& C\beta_b^2\left( {\vert\vert I_h(\tilde{b}) \vert\vert}_2^2+{\vert\vert I_h(b) \vert\vert}_2^2 \right)
+\frac{1}{4}\mu{\vert\vert \Delta \tilde{b} \vert\vert}_2^2.
\end{align*}
By the above estimates, we have
\begin{align*}
&\frac{d}{dt}\left({\vert\vert \nabla \tilde{u} \vert\vert}_2^2+{\vert\vert \nabla \tilde{b} \vert\vert}_2^2\right)
+\mu\left({\vert\vert \Delta{\tilde{u}} \vert\vert}_2^2+{\vert\vert \Delta{\tilde{b}} \vert\vert}_2^2\right)
+\nu\left({\vert\vert \nabla\partial_z \tilde{u} \vert\vert}_2^2+{\vert\vert \nabla\partial_z \tilde{b} \vert\vert}_2^2\right)\\
\leq& C\left({\vert\vert \tilde{u} \vert\vert}_2^2{\vert\vert \nabla \tilde{u} \vert\vert}_2^2+{\vert\vert \partial_z \tilde{u} \vert\vert}_2^2+{\vert\vert \tilde{b} \vert\vert}_2^2{\vert\vert \nabla \tilde{b} \vert\vert}_2^2+{\vert\vert \partial_z \tilde{b} \vert\vert}_2^2\right)
\left({\vert\vert \nabla \tilde{u} \vert\vert}_2^2+{\vert\vert \nabla \tilde{b} \vert\vert}_2^2\right).
\end{align*}
Uing Gronwall's inequality, we get that
$\tilde{u}$ and $\tilde{b}$ are $H^1$ uniformly bounded.
Integrating in time $t$ to $t+1$, we can get the uniform boundedness of the time average of
${\vert\vert \tilde{u} \vert\vert}_{H^2}^2$ and ${\vert\vert \tilde{b} \vert\vert}_{H^2}^2$.
Thus
\begin{align}
{\vert\vert \tilde{u} \vert\vert}_{H^1}^2+{\vert\vert \tilde{ b} \vert\vert}_{H^1}^2
+\int_t^{t+1} {\left({\vert\vert \tilde{u} \vert\vert}_{H^2}^2+{\vert\vert \tilde{b} \vert\vert}_{H^2}^2\right)} \,{\rm d}s \leq k_0.
\end{align}

\subsection{$H^2$-uniform estimates for $u,b,\tilde{u},\tilde{b}$}

Taking the inner product of equation \eqref{2.1} and \eqref{2.3} with $\Delta^2 u$ and $\Delta^2 b$ respectively, we have
\begin{align*}
&\frac{1}{2}\frac{d}{dt}\left({\vert\vert \Delta u \vert\vert}_2^2+{\vert\vert \Delta b \vert\vert}_2^2\right)
+\mu\left({\vert\vert \nabla\Delta{u} \vert\vert}_2^2+{\vert\vert \nabla\Delta{b} \vert\vert}_2^2\right)
+\nu\left({\vert\vert \partial_z\Delta{u} \vert\vert}_2^2+{\vert\vert \partial_z\Delta{b} \vert\vert}_2^2\right)\\
\leq& \int_\Omega {\nabla[(u\cdot\nabla)u+u_3\partial_z{u}]:\nabla\Delta u} \,{\rm d}\Omega
+\int_\Omega {\nabla\left[(b\cdot\nabla)b+u_3\partial_z{b}\right]:\nabla\Delta u} \,{\rm d}\Omega\\
&+\int_\Omega {\nabla\left[(u\cdot\nabla)b+b_3\partial_z{b}\right]:\nabla\Delta b} \,{\rm d}\Omega
+\int_\Omega {\nabla\left[(b\cdot\nabla)u+b_3\partial_z{u}\right]:\nabla\Delta b} \,{\rm d}\Omega\\
=&:N_1+N_2+N_3+N_4,
\end{align*}
where
\begin{align*}
N_1 &= \int_\Omega {\nabla[(u\cdot\nabla)u+u_3\partial_z{u}]:\nabla\Delta u} \,{\rm d}\Omega\\
&=\int_\Omega {\nabla[(u\cdot\nabla)u]:\nabla\Delta u} \,{\rm d}\Omega
+\int_\Omega {\nabla(u_3\partial_z{u}):\nabla\Delta u} \,{\rm d}\Omega\\
&=:N_{11}+N_{12}.
\end{align*}
Using Lemma \ref{lem2.1}, the Poincar\'{e} and Young inequalities, we infer that
\begin{align*}
N_{11} &=\int_\Omega {\nabla[(u\cdot\nabla)u]:\nabla\Delta u} \,{\rm d}\Omega\\
&\leq \int_\Omega {\left[(\partial_i u\cdot\nabla)u+(u\cdot\partial_i\nabla)u\right]\partial_i\Delta u} \,{\rm d}\Omega\\
&\leq {\frac{1}{16}\mu}{\vert\vert \nabla\Delta u \vert\vert}_2^2
+C{\vert\vert \nabla u \vert\vert}_2^2{\vert\vert \Delta u \vert\vert}_2^4,\\
N_{12} &= \int_\Omega {\nabla(u_3\partial_z{u}):\nabla\Delta u} \,{\rm d}\Omega\\
&\leq \int_M {\left(\int_{-1}^1 {\vert \Delta u \vert} \,{\rm d}z\right)
\left(\int_{-1}^1 {{\vert \partial_z\nabla u \vert}{\vert \nabla\Delta u \vert}} \,{\rm d}z\right)} \,{\rm d}M\\
&\leq C{\vert\vert \Delta u \vert\vert}_2^{\frac{1}{2}}\left({\vert\vert \Delta u \vert\vert}_2+{\vert\vert \nabla\Delta u \vert\vert}_2\right)^{\frac{1}{2}}{\vert\vert \partial_z\nabla u \vert\vert}_2^{\frac{1}{2}}\left({\vert\vert \partial_z\nabla u \vert\vert}_2+{\vert\vert \partial_z\Delta u \vert\vert}_2\right)^{\frac{1}{2}}{\vert\vert \nabla\Delta u \vert\vert}_2\\
&\leq {\frac{1}{16}\mu}{\vert\vert \nabla\Delta u \vert\vert}_2^2+{\frac{1}{4}}\nu{\vert\vert \partial_z\Delta u \vert\vert}_2^2
+C{\vert\vert \partial_z\nabla u \vert\vert}_2^2{\vert\vert \Delta u \vert\vert}_2^2.
\end{align*}
Therefore, we have
\begin{align*}
N_{1} \leq& {\frac{1}{8}}\mu{\vert\vert \nabla\Delta u \vert\vert}_2^2+{\frac{1}{4}}\nu{\vert\vert \partial_z\Delta u \vert\vert}_2^2
+C\left({\vert\vert \nabla u \vert\vert}_2^2{\vert\vert \Delta u \vert\vert}_2^2+{\vert\vert \partial_z\nabla u \vert\vert}_2^2\right){\vert\vert \Delta u \vert\vert}_2^2.
\end{align*}
Similarly, we can get
\begin{align*}
N_{2} \leq& {\frac{1}{8}}\mu{\vert\vert \nabla\Delta u \vert\vert}_2^2+{\frac{1}{10}}\mu{\vert\vert \nabla\Delta b \vert\vert}_2^2+{\frac{1}{4}}\nu{\vert\vert \partial_z\Delta b \vert\vert}_2^2\\
&+C\left({\vert\vert \nabla b \vert\vert}_2^2{\vert\vert \Delta b \vert\vert}_2^4+{\vert\vert \partial_z\nabla b \vert\vert}_2^2{\vert\vert \Delta u \vert\vert}_2^2\right),\\
N_{3} \leq& {\frac{1}{8}}\mu{\vert\vert \nabla\Delta u \vert\vert}_2^2+{\frac{1}{5}}\mu{\vert\vert \nabla\Delta b \vert\vert}_2^2+{\frac{1}{4}}\nu{\vert\vert \partial_z\Delta b \vert\vert}_2^2\\
&+C\left({\vert\vert \nabla u \vert\vert}_2^2{\vert\vert \Delta u \vert\vert}_2^2+{\vert\vert \nabla b \vert\vert}_2^2{\vert\vert \Delta u \vert\vert}_2^2+{\vert\vert \partial_z\nabla b \vert\vert}_2^2\right){\vert\vert \Delta b \vert\vert}_2^2,\\
N_{4} \leq& {\frac{1}{8}}\mu{\vert\vert \nabla\Delta u \vert\vert}_2^2+{\frac{1}{5}}\mu{\vert\vert \nabla\Delta b \vert\vert}_2^2+{\frac{1}{4}}\nu{\vert\vert \partial_z\Delta u \vert\vert}_2^2\\
&+C\left({\vert\vert \nabla u \vert\vert}_2^2{\vert\vert \Delta u \vert\vert}_2^2+{\vert\vert \nabla b \vert\vert}_2^2{\vert\vert \Delta u \vert\vert}_2^2+{\vert\vert \partial_z\nabla u \vert\vert}_2^2\right){\vert\vert \Delta b \vert\vert}_2^2.
\end{align*}
Thus
\begin{align*}
&\frac{d}{dt}\left({\vert\vert \Delta u \vert\vert}_2^2+{\vert\vert \Delta b \vert\vert}_2^2\right)
+\mu\left({\vert\vert \nabla\Delta{u} \vert\vert}_2^2+{\vert\vert \nabla\Delta{b} \vert\vert}_2^2\right)
+\nu\left({\vert\vert \partial_z\Delta{u} \vert\vert}_2^2+{\vert\vert \partial_z\Delta{b} \vert\vert}_2^2\right) \\
\leq& C\left({\vert\vert \nabla u \vert\vert}_2^2{\vert\vert \Delta u \vert\vert}_2^2+{\vert\vert \nabla b \vert\vert}_2^2{\vert\vert \Delta b \vert\vert}_2^2+{\vert\vert \partial_z\nabla u \vert\vert}_2^2+{\vert\vert \partial_z\nabla b \vert\vert}_2^2\right)
\left({\vert\vert \Delta u \vert\vert}_2^2+{\vert\vert \Delta b \vert\vert}_2^2\right).
\end{align*}
Using Gronwall's inequality, we get that
$u$ and $b$ are $H^2$ uniformly bounded.
Integrating in time $t$ to $t+1$, we can get the uniform boundedness of the time average of
${\vert\vert u \vert\vert}_{H^3}^2$ and ${\vert\vert b \vert\vert}_{H^3}^2$.
Thus
\begin{align}
{\vert\vert u \vert\vert}_{H^2}^2+{\vert\vert b \vert\vert}_{H^2}^2
+\int_t^{t+1} {\left({\vert\vert u \vert\vert}_{H^3}^2+{\vert\vert b \vert\vert}_{H^3}^2\right)} \,{\rm d}s \leq k_0.
\end{align}

Similarly, taking the inner product of equation \eqref{2.11} and \eqref{2.13} with $\Delta^2 \tilde{u}$ and $\Delta^2 \tilde{b}$ respectively, we have
\begin{align*}
&\frac{1}{2}\frac{d}{dt}\left({\vert\vert \Delta \tilde{u} \vert\vert}_2^2+{\vert\vert \Delta \tilde{b} \vert\vert}_2^2\right)
+\mu\left({\vert\vert \nabla\Delta{\tilde{u}} \vert\vert}_2^2+{\vert\vert \nabla\Delta{\tilde{b}} \vert\vert}_2^2\right)
+\nu\left({\vert\vert \partial_z\Delta{\tilde{u}} \vert\vert}_2^2+{\vert\vert \partial_z\Delta{\tilde{b}} \vert\vert}_2^2\right)\\
\leq& \int_\Omega {\nabla[(\tilde{u}\cdot\nabla)\tilde{u}+\tilde{u}_3\partial_z{\tilde{u}}]:\nabla\Delta \tilde{u}} \,{\rm d}\Omega
+\int_\Omega {\nabla\left[(\tilde{b}\cdot\nabla)\tilde{b}+\tilde{u}_3\partial_z{\tilde{b}}\right]:\nabla\Delta \tilde{u}} \,{\rm d}\Omega\\
&+\int_\Omega {\nabla\left[(\tilde{u}\cdot\nabla)\tilde{b}+\tilde{b}_3\partial_z{\tilde{b}}\right]:\nabla\Delta \tilde{b}} \,{\rm d}\Omega
+\int_\Omega {\nabla\left[(\tilde{b}\cdot\nabla)\tilde{u}+\tilde{b}_3\partial_z{\tilde{u}}\right]:\nabla\Delta \tilde{b}} \,{\rm d}\Omega\\
&-\int_\Omega {-\beta_u\left(I_h(\tilde{u})-I_h(u)\right)(-\Delta \tilde{u})} \,{\rm d}\Omega-\int_\Omega {-\beta_b\left(I_h(\tilde{b})-I_h(b)\right)(-\Delta \tilde{b})} \,{\rm d}\Omega.
\end{align*}
We can get the key estimates
\begin{align*}
\int_\Omega {\beta_u\left(I_h(\tilde{u})-I_h(u)\right)\Delta^2 \tilde{u}} \,{\rm d}\Omega
\leq& C\beta_u\left( {\vert\vert \nabla I_h(\tilde{u}) \vert\vert}_2+{\vert\vert \nabla I_h(u) \vert\vert}_2 \right)
{\vert\vert \nabla\Delta \tilde{u} \vert\vert}_2 \\
\leq& C\beta_u^2\left( {\vert\vert \nabla \tilde{u} \vert\vert}_2^2+{\vert\vert \nabla u \vert\vert}_2^2 \right)
+{\frac{1}{4}}\mu{\vert\vert \nabla\Delta \tilde{u} \vert\vert}_2^2,
\end{align*}
and
\begin{align*}
\int_\Omega {\beta_b\left(I_h(\tilde{b})-I_h(b)\right)\Delta^2 \tilde{b}} \,{\rm d}\Omega
\leq& C\beta_b\left( {\vert\vert \nabla I_h(\tilde{b}) \vert\vert}_2+{\vert\vert \nabla I_h(b) \vert\vert}_2 \right)
{\vert\vert \nabla\Delta \tilde{b} \vert\vert}_2 \\
\leq& C\beta_b^2\left( {\vert\vert \nabla \tilde{b} \vert\vert}_2^2+{\vert\vert \nabla b \vert\vert}_2^2 \right)
+{\frac{1}{4}}\mu{\vert\vert \nabla\Delta \tilde{b} \vert\vert}_2^2.
\end{align*}
By the above estimates, we have
\begin{align}
{\vert\vert \tilde{u} \vert\vert}_{H^2}^2+{\vert\vert \tilde{b} \vert\vert}_{H^2}^2
+\int_t^{t+1} {\left({\vert\vert \tilde{u} \vert\vert}_{H^3}^2+{\vert\vert \tilde{b} \vert\vert}_{H^3}^2\right)} \,{\rm d}s \leq k_0.
\end{align}

Using the uniformly bounded estimates obtained previously, we follow similar arguments as in \cite[Proposition 3.6]{pe3} to show that $\vert\vert (u_t,b_t) \vert\vert_{H^1}$ and $\vert\vert (\tilde{u}_t,\tilde{b}_t) \vert\vert_{H^1}$ are also uniformly bounded. Hence, following the argument in \cite{pe5}, we can get the existence of a global attractor for both the PEM and its assimilation system. Moreover, we derive the following two propositions.

\begin{proposition}
Let $({u}_0, {b}_0)\in H^2$. Then system \eqref{2.1}-\eqref{2.4} has a unique strong solution satisfying that
\begin{align*}
&({u},{b})\in C(0,+\infty;H^2)\cap L^2(0,+\infty;H^3),\\
&(\partial_t {u},\partial_t {b})\in L^2(0,+\infty;H^1).
\end{align*}
Moreover, the semigroup generated by problem \eqref{2.1}-\eqref{2.4} has a global attractor.
\end{proposition}

\begin{proposition}
Let $(\tilde{u}_0, \tilde{b}_0)\in H^2$ and $\{(C_h^1)^2\beta_uh^2,(C_h^1)^2\beta_bh^2\}\leq\min\{\mu,\nu\}$. Then the assimilation system \eqref{2.11}-\eqref{2.14} has a unique strong solution satisfying that
\begin{align*}
&(\tilde{u},\tilde{b})\in C(0,+\infty;H^2)\cap L^2(0,+\infty;H^3),\\
&(\partial_t \tilde{u},\partial_t \tilde{b})\in L^2(0,+\infty;H^1).
\end{align*}
Moreover, the semigroup generated by problem \eqref{2.11}-\eqref{2.14} has a global attractor.
\end{proposition}

Next, we consider the convergence relation between the assimilation solution and the reference solution.

\begin{theorem}
Let $(u,b)$ and $(\tilde{u},\tilde{b})$ be the unique solution of the PEM system \eqref{2.1}-\eqref{2.4} and its assimilation system \eqref{2.11}-\eqref{2.14}, respectively. Suppose that $\beta_u,\beta_b\; and\; h$ satisfy
$$
\beta_u \geq 2(\Re_1+\Re_2+\Re_3),\;
\beta_b \geq 2(\Re_1+\Re_2+\Re_3),
$$
and
$$
0 < h \leq \min\left\{\frac{1}{\beta_u C_h^1}\sqrt{\frac{\mu\nu}{\mu+4\nu}},\;
\frac{1}{\beta_b C_h^1}\sqrt{\frac{\mu\nu}{\mu+4\nu}}\right\},
$$
then $(\tilde{u},\tilde{b})$ converges to $(u,b)$ exponentially in $L^2$ as $t\rightarrow\infty$. Here $\Re_1,\Re_2,\Re_3$ are constants defined in \eqref{r1}-\eqref{r3} below.
\end{theorem}
\begin{proof}
Set $$\hat{u}=u-\tilde{u}, \hat{u}_3=u_3-\tilde{u}_3, \hat{b}=b-\tilde{b}, \hat{b}_3=b_3-\tilde{b}_3, \hat{p}=p-\tilde{p}.$$
Subtracting the assimilation equations from the PEM equations, we have
\begin{align}
\partial_t&{\hat{u}}+B(\hat{u},u)+B(u,\hat{u})-B(\hat{u},\hat{u})+\nabla \hat{p}\nonumber\\
&-B(\hat{b},b)-B(b,\hat{b})+B(\hat{b},\hat{b})-\mu\Delta\hat{u}-\nu{\partial_z^2}\hat{u}=-\beta_u {I_h(\hat{u})},\label{3.19}\\
\partial_z&{\hat{p}}=0,\\
\partial_t&{\hat{b}}+B(\hat{u},b)+B(u,\hat{b})-B(\hat{u},\hat{b})\nonumber\\
&-B(\hat{b},u)-B(b,\hat{u})+B(\hat{b},\hat{u})-\kappa\Delta\hat{b}-\sigma{\partial_z^2}\hat{b}=-\beta_b {I_h(\hat{b})},\label{3.21}\\
\nabla&\cdot{\hat{u}}+\partial_z{{\hat{u}}_3}=0,\; \nabla\cdot{\hat{b}}+\partial_z{{\hat{b}}_3}=0.
\end{align}
Taking the inner product of equation \eqref{3.19} and \eqref{3.21} with $\hat{u}$ and $\hat{b}$ respectively, we have
\begin{align*}
&\frac{1}{2}\frac{d}{dt}\left({\vert\vert \hat{u} \vert\vert}_2^2+{\vert\vert \hat{b} \vert\vert}_2^2\right)
+\mu{\vert\vert \nabla{\hat{u}} \vert\vert}_2^2+\nu{\vert\vert \partial_z{\hat{u}} \vert\vert}_2^2
+\kappa{\vert\vert \nabla{\hat{b}} \vert\vert}_2^2+\sigma{\vert\vert \partial_z{\hat{b}} \vert\vert}_2^2\\
=&-\int_\Omega {\nabla \hat{p} \cdot \hat{u}} \,{\rm d}\Omega-\beta_u\int_\Omega {I_h(\hat{u}) \hat{u}} \,{\rm d}\Omega
-\beta_b\int_\Omega {I_h(\hat{b}) \hat{b}} \,{\rm d}\Omega\\
&-\int_\Omega {\left[(\hat{u}\cdot\nabla)u+\hat{u}_3\partial_zu\right] \hat{u}} \,{\rm d}\Omega
+\int_\Omega {\left[(\hat{b}\cdot\nabla)b+\hat{b}_3\partial_zb\right] \hat{u}} \,{\rm d}\Omega\\
&-\int_\Omega {\left[(\hat{u}\cdot\nabla)b+\hat{u}_3\partial_zb\right] \hat{b}} \,{\rm d}\Omega
+\int_\Omega {\left[(\hat{b}\cdot\nabla)u+\hat{b}_3\partial_zu\right] \hat{b}} \,{\rm d}\Omega\\
=&: J_{1}+J_{2}+J_{3}+I_{1}+I_{2}+I_{3}+I_{4}.
\end{align*}
Using Lemma \ref{lem2.1}, we infer that
\begin{align*}
I_{1} =& \int_\Omega {\left[(\hat{u}\cdot\nabla)u+\hat{u}_3\partial_zu\right] \hat{u}} \,{\rm d}\Omega\\
\leq& \int_M {\left(\int_{-1}^1 {(\vert \hat{u} \vert+\vert \partial_z \hat{u} \vert)} \,{\rm d}z\right)
\left(\int_{-1}^1 {\vert \nabla u \vert \vert \hat{u} \vert} \,{\rm d}z\right)} \,{\rm d}M\\
&+ \int_M {\left(\int_{-1}^1 {\nabla \hat{u}} \,{\rm d}z\right)\left(\int_{-1}^1 {\vert \partial_z u \vert \vert \hat{u} \vert} \,{\rm d}z\right)} \,{\rm d}M\\
\leq& C{\vert\vert \hat{u} \vert\vert}_2^{\frac{1}{2}}\left({\vert\vert  \hat{u} \vert\vert}_2^{\frac{1}{2}}+{\vert\vert \nabla \hat{u} \vert\vert}_2^{\frac{1}{2}}\right)
{\vert\vert \nabla u \vert\vert}_2
{\vert\vert \hat{u} \vert\vert}_2^{\frac{1}{2}}\left({\vert\vert \hat{u} \vert\vert}_2^{\frac{1}{2}}+{\vert\vert \nabla \hat{u} \vert\vert}_2^{\frac{1}{2}}\right)\\
&+C{\vert\vert \partial_z\hat{u} \vert\vert}_2
{\vert\vert \nabla u \vert\vert}_2^{\frac{1}{2}}\left({\vert\vert  \nabla u \vert\vert}_2^{\frac{1}{2}}+{\vert\vert \Delta u \vert\vert}_2^{\frac{1}{2}}\right)
\left({\vert\vert \hat{u} \vert\vert}_2+{\vert\vert  \hat{u} \vert\vert}_2^{\frac{1}{2}}{\vert\vert \nabla \hat{u} \vert\vert}_2^{\frac{1}{2}}\right)\\
&+C{\vert\vert \nabla\hat{u} \vert\vert}_2
{\vert\vert \partial_z u \vert\vert}_2^{\frac{1}{2}}\left({\vert\vert  \partial_z u \vert\vert}_2^{\frac{1}{2}}+{\vert\vert \nabla\partial_z u \vert\vert}_2^{\frac{1}{2}}\right)
\left({\vert\vert \hat{u} \vert\vert}_2+{\vert\vert  \hat{u} \vert\vert}_2^{\frac{1}{2}}{\vert\vert \nabla \hat{u} \vert\vert}_2^{\frac{1}{2}}\right)\\
=&:I_{11}+I_{12}+I_{13}.
\end{align*}
Using the Young inequality, we infer that
\begin{align*}
I_{11} & \leq Ck_0{\vert\vert \hat{u} \vert\vert}_2\left({\vert\vert \hat{u} \vert\vert}_2+{\vert\vert \nabla\hat{u} \vert\vert}_2+{\vert\vert \hat{u} \vert\vert}_2^{\frac{1}{2}}{\vert\vert \nabla\hat{u} \vert\vert}_2^{\frac{1}{2}}\right)\\
& \leq Ck_0{\vert\vert \hat{u} \vert\vert}_2^2+Ck_0{\vert\vert \hat{u} \vert\vert}_2{\vert\vert \nabla\hat{u} \vert\vert}_2+2Ck_0{\vert\vert \hat{u} \vert\vert}_2^{\frac{1}{2}}{\vert\vert \nabla\hat{u} \vert\vert}_2^{\frac{1}{2}}\\
& \leq \left[Ck_0+\frac{4C^2k_0^2}{\mu}+\frac{3}{4}(2Ck_0)^\frac{4}{3}\left(\frac{4}{\mu}\right)^\frac{1}{3}\right]{\vert\vert \hat{u} \vert\vert}_2^2+\frac{\mu}{8}{\vert\vert \nabla\hat{u} \vert\vert}_2^2,\\
I_{12} & \leq Ck_0{\vert\vert \partial_z\hat{u} \vert\vert}_2\left({\vert\vert \hat{u} \vert\vert}_2+{\vert\vert  \hat{u} \vert\vert}_2^{\frac{1}{2}}{\vert\vert \nabla \hat{u} \vert\vert}_2^{\frac{1}{2}}\right)\\
& \leq \frac{\nu}{4}{\vert\vert \partial_z\hat{u} \vert\vert}_2^2+\frac{2C^2k_0^2}{\nu}\left({\vert\vert \hat{u} \vert\vert}_2+{\vert\vert  \hat{u} \vert\vert}_2{\vert\vert \nabla \hat{u} \vert\vert}_2\right)\\
& \leq \left(\frac{2C^2k_0^2}{\nu}+\frac{16C^4k_0^4}{\mu\nu^2}\right){\vert\vert \hat{u} \vert\vert}_2^2+\frac{\nu}{4}{\vert\vert \partial_z\hat{u} \vert\vert}_2^2+\frac{\mu}{16}{\vert\vert \nabla\hat{u} \vert\vert}_2^2,\\
I_{13} & \leq Ck_0{\vert\vert \nabla \hat{u} \vert\vert}_2{\vert\vert \hat{u} \vert\vert}_2+Ck_0{\vert\vert \nabla \hat{u} \vert\vert}_2^\frac{3}{2}{\vert\vert \hat{u} \vert\vert}_2^\frac{1}{2}\\
& \leq \frac{\mu}{8}{\vert\vert \nabla \hat{u} \vert\vert}_2^2+\left(\frac{4C^2k_0^2}{\mu}+\frac{432C^4k_0^4}{\mu^3}\right){\vert\vert \hat{u} \vert\vert}_2^2.
\end{align*}
Thus
\begin{align*}
I_{1}\leq& \frac{5\mu}{16}{\vert\vert \nabla \hat{u} \vert\vert}_2^2+\frac{\nu}{4}{\vert\vert \partial_z\hat{u} \vert\vert}_2^2\\
&+\left[Ck_0+\frac{8C^2k_0^2}{\mu}+\frac{3}{4}(2Ck_0)^\frac{4}{3}\left(\frac{4}{\mu}\right)^\frac{1}{3}+
\frac{2C^2k_0^2}{\nu}+\frac{16C^4k_0^4}{\mu\nu^2}+\frac{432C^4k_0^4}{\mu^3}\right]{\vert\vert \hat{u} \vert\vert}_2^2\\
=&: \frac{5\mu}{16}{\vert\vert \nabla \hat{u} \vert\vert}_2^2+\frac{\nu}{4}{\vert\vert \partial_z\hat{u} \vert\vert}_2^2+\Re_1{\vert\vert \hat{u} \vert\vert}_2^2.
\end{align*}
Using Lemma \ref{lem2.1} and the Young inequality, we infer that
$$
\begin{aligned}
I_{2} =& \int_\Omega {\left[(\hat{b}\cdot\nabla)b+\hat{b}_3\partial_zb\right] \hat{u}} \,{\rm d}\Omega\\
\leq& \int_M {\left(\int_{-1}^1 {(\vert \hat{b} \vert+\vert \partial_z \hat{b} \vert)} \,{\rm d}z\right)
\left(\int_{-1}^1 {\vert \nabla b \vert \vert \hat{u} \vert} \,{\rm d}z\right)} \,{\rm d}M\\
&+ \int_M {\left(\int_{-1}^1 {\nabla \hat{b}} \,{\rm d}z\right)\left(\int_{-1}^1 {\vert \partial_z b \vert \vert \hat{u} \vert} \,{\rm d}z\right)} \,{\rm d}M\\
\leq& C{\vert\vert \hat{b} \vert\vert}_2^{\frac{1}{2}}\left({\vert\vert  \hat{b} \vert\vert}_2^{\frac{1}{2}}+{\vert\vert \nabla \hat{b} \vert\vert}_2^{\frac{1}{2}}\right)
{\vert\vert \nabla b \vert\vert}_2
{\vert\vert \hat{u} \vert\vert}_2^{\frac{1}{2}}\left({\vert\vert \hat{u} \vert\vert}_2^{\frac{1}{2}}+{\vert\vert \nabla \hat{u} \vert\vert}_2^{\frac{1}{2}}\right)\\
&+C{\vert\vert \partial_z\hat{b} \vert\vert}_2
{\vert\vert \nabla b \vert\vert}_2^{\frac{1}{2}}\left({\vert\vert  \nabla b \vert\vert}_2^{\frac{1}{2}}+{\vert\vert \Delta b \vert\vert}_2^{\frac{1}{2}}\right)
\left({\vert\vert \hat{u} \vert\vert}_2+{\vert\vert  \hat{u} \vert\vert}_2^{\frac{1}{2}}{\vert\vert \nabla \hat{u} \vert\vert}_2^{\frac{1}{2}}\right)\\
&+C{\vert\vert \nabla\hat{b} \vert\vert}_2
{\vert\vert \partial_z b \vert\vert}_2^{\frac{1}{2}}\left({\vert\vert  \partial_z b \vert\vert}_2^{\frac{1}{2}}+{\vert\vert \nabla\partial_z b \vert\vert}_2^{\frac{1}{2}}\right)
\left({\vert\vert \hat{u} \vert\vert}_2+{\vert\vert  \hat{u} \vert\vert}_2^{\frac{1}{2}}{\vert\vert \nabla \hat{u} \vert\vert}_2^{\frac{1}{2}}\right)\\
=&:I_{21}+I_{22}+I_{23}.
\end{aligned}
$$
Using the Young inequality, we infer that
\begin{align*}
I_{21} & \leq Ck_0\left({\vert\vert \hat{b} \vert\vert}_2{\vert\vert \hat{u} \vert\vert}_2+{\vert\vert \hat{b} \vert\vert}_2{\vert\vert \hat{u} \vert\vert}_2^\frac{1}{2}{\vert\vert \nabla\hat{u} \vert\vert}_2^\frac{1}{2}+
{\vert\vert \hat{b} \vert\vert}_2^\frac{1}{2}{\vert\vert \hat{u} \vert\vert}_2{\vert\vert \nabla\hat{b} \vert\vert}_2^\frac{1}{2}
+{\vert\vert \hat{b} \vert\vert}_2^\frac{1}{2}{\vert\vert \nabla\hat{b} \vert\vert}_2^\frac{1}{2}
{\vert\vert \hat{u} \vert\vert}_2^\frac{1}{2}{\vert\vert \nabla\hat{u} \vert\vert}_2^\frac{1}{2}\right)\\
& \leq \left(Ck_0+\frac{2C^2k_0^2}{\mu}\right){\vert\vert \hat{u} \vert\vert}_2^2
+\left(Ck_0+\frac{2C^2k_0^2}{\mu}\right){\vert\vert \hat{b} \vert\vert}_2^2
+\frac{\mu}{8}{\vert\vert \nabla\hat{u} \vert\vert}_2^2+\frac{\mu}{8}{\vert\vert \nabla\hat{b} \vert\vert}_2^2,\\
I_{22} & \leq Ck_0\left({\vert\vert \partial_z\hat{b} \vert\vert}_2{\vert\vert \hat{u} \vert\vert}_2
+{\vert\vert \partial_z\hat{b} \vert\vert}_2{\vert\vert \hat{u} \vert\vert}_2^\frac{1}{2}{\vert\vert \nabla\hat{u} \vert\vert}_2^\frac{1}{2}\right)\\
& \leq\left(\frac{2C^2k_0^2}{\nu}+\frac{16C^4k_0^4}{\mu\nu^2}\right){\vert\vert \hat{u} \vert\vert}_2^2
+\frac{\nu}{4}{\vert\vert \partial_z\hat{b} \vert\vert}_2^2
+\frac{\mu}{16}{\vert\vert \nabla\hat{u} \vert\vert}_2^2,\\
I_{23} & \leq Ck_0\left({\vert\vert \nabla\hat{b} \vert\vert}_2{\vert\vert \hat{u} \vert\vert}_2
+{\vert\vert \nabla\hat{b} \vert\vert}_2{\vert\vert \hat{u} \vert\vert}_2^\frac{1}{2}{\vert\vert \nabla\hat{u} \vert\vert}_2^\frac{1}{2}\right)\\
& \leq\left(\frac{4C^2k_0^2}{\mu}+\frac{64C^4k_0^4}{\mu^3}\right){\vert\vert \hat{u} \vert\vert}_2^2
+\frac{\mu}{8}{\vert\vert \nabla\hat{b} \vert\vert}_2^2
+\frac{\mu}{16}{\vert\vert \nabla\hat{u} \vert\vert}_2^2.
\end{align*}
Thus
\begin{align*}
I_{2}\leq& \left(Ck_0+\frac{6C^2k_0^2}{\mu}+\frac{2C^2k_0^2}{\nu}+\frac{16C^4k_0^4}{\mu\nu^2}
+\frac{64C^4k_0^4}{\mu^3}\right){\vert\vert \hat{u} \vert\vert}_2^2
+\left(Ck_0+\frac{2C^2k_0^2}{\mu}\right){\vert\vert \hat{b} \vert\vert}_2^2\\
&+\frac{\nu}{4}{\vert\vert \partial_z\hat{b} \vert\vert}_2^2
+\frac{\mu}{4}{\vert\vert \nabla\hat{u} \vert\vert}_2^2
+\frac{\mu}{4}{\vert\vert \nabla\hat{b} \vert\vert}_2^2\\
=&: \Re_2{\vert\vert \hat{u} \vert\vert}_2^2
+\Re_3{\vert\vert \hat{b} \vert\vert}_2^2
+\frac{\nu}{4}{\vert\vert \partial_z\hat{b} \vert\vert}_2^2
+\frac{\mu}{4}{\vert\vert \nabla\hat{u} \vert\vert}_2^2
+\frac{\mu}{4}{\vert\vert \nabla\hat{b} \vert\vert}_2^2.
\end{align*}
Using Lemma \ref{lem2.1} and the Young inequality, we infer that
$$
\begin{aligned}
I_{3} &= \int_\Omega {\left[(\hat{u}\cdot\nabla)b+\hat{u}_3\partial_z b\right] \hat{b}} \,{\rm d}\Omega\\
\leq& \int_M {\left(\int_{-1}^1 {(\vert \hat{u} \vert+\vert \partial_z \hat{u} \vert)} \,{\rm d}z\right)
\left(\int_{-1}^1 {\vert \nabla b \vert \vert \hat{b} \vert} \,{\rm d}z\right)} \,{\rm d}M\\
&+ \int_M {\left(\int_{-1}^1 {\nabla \hat{u}} \,{\rm d}z\right)\left(\int_{-1}^1 {\vert \partial_z b \vert \vert \hat{b} \vert} \,{\rm d}z\right)} \,{\rm d}M\\
\leq& C{\vert\vert \hat{u} \vert\vert}_2^{\frac{1}{2}}\left({\vert\vert  \hat{u} \vert\vert}_2^{\frac{1}{2}}+{\vert\vert \nabla \hat{u} \vert\vert}_2^{\frac{1}{2}}\right)
{\vert\vert \nabla b \vert\vert}_2
{\vert\vert \hat{b} \vert\vert}_2^{\frac{1}{2}}\left({\vert\vert \hat{b} \vert\vert}_2^{\frac{1}{2}}+{\vert\vert \nabla \hat{b} \vert\vert}_2^{\frac{1}{2}}\right)\\
&+C{\vert\vert \partial_z\hat{u} \vert\vert}_2
{\vert\vert \nabla b \vert\vert}_2^{\frac{1}{2}}\left({\vert\vert  \nabla b \vert\vert}_2^{\frac{1}{2}}+{\vert\vert \Delta b \vert\vert}_2^{\frac{1}{2}}\right)
\left({\vert\vert \hat{b} \vert\vert}_2+{\vert\vert  \hat{b} \vert\vert}_2^{\frac{1}{2}}{\vert\vert \nabla \hat{b} \vert\vert}_2^{\frac{1}{2}}\right)\\
&+C{\vert\vert \nabla\hat{u} \vert\vert}_2
{\vert\vert \partial_z b \vert\vert}_2^{\frac{1}{2}}\left({\vert\vert  \partial_z b \vert\vert}_2^{\frac{1}{2}}+{\vert\vert \nabla\partial_z b \vert\vert}_2^{\frac{1}{2}}\right)
\left({\vert\vert \hat{b} \vert\vert}_2+{\vert\vert  \hat{b} \vert\vert}_2^{\frac{1}{2}}{\vert\vert \nabla \hat{b} \vert\vert}_2^{\frac{1}{2}}\right)\\
=&:I_{31}+I_{32}+I_{33}.
\end{aligned}
$$
Using the Young inequality, we infer that
\begin{align*}
I_{31} & \leq Ck_0\left({\vert\vert \hat{b} \vert\vert}_2{\vert\vert \hat{u} \vert\vert}_2+{\vert\vert \hat{b} \vert\vert}_2{\vert\vert \hat{u} \vert\vert}_2^\frac{1}{2}{\vert\vert \nabla\hat{u} \vert\vert}_2^\frac{1}{2}+
{\vert\vert \hat{b} \vert\vert}_2^\frac{1}{2}{\vert\vert \hat{u} \vert\vert}_2{\vert\vert \nabla\hat{b} \vert\vert}_2^\frac{1}{2}
+{\vert\vert \hat{b} \vert\vert}_2^\frac{1}{2}{\vert\vert \nabla\hat{b} \vert\vert}_2^\frac{1}{2}
{\vert\vert \hat{u} \vert\vert}_2^\frac{1}{2}{\vert\vert \nabla\hat{u} \vert\vert}_2^\frac{1}{2}\right)\\
& \leq \left(Ck_0+\frac{2C^2k_0^2}{\mu}\right){\vert\vert \hat{u} \vert\vert}_2^2
+\left(Ck_0+\frac{2C^2k_0^2}{\mu}\right){\vert\vert \hat{b} \vert\vert}_2^2
+\frac{\mu}{8}{\vert\vert \nabla\hat{u} \vert\vert}_2^2+\frac{\mu}{8}{\vert\vert \nabla\hat{b} \vert\vert}_2^2,\\
I_{32} & \leq Ck_0\left({\vert\vert \partial_z\hat{u} \vert\vert}_2{\vert\vert \hat{b} \vert\vert}_2
+{\vert\vert \partial_z\hat{u} \vert\vert}_2{\vert\vert \hat{b} \vert\vert}_2^\frac{1}{2}{\vert\vert \nabla\hat{b} \vert\vert}_2^\frac{1}{2}\right)\\
& \leq\left(\frac{2C^2k_0^2}{\nu}+\frac{16C^4k_0^4}{\mu\nu^2}\right){\vert\vert \hat{b} \vert\vert}_2^2
+\frac{\nu}{4}{\vert\vert \partial_z\hat{u} \vert\vert}_2^2
+\frac{\mu}{16}{\vert\vert \nabla\hat{b} \vert\vert}_2^2,\\
I_{33} & \leq Ck_0\left({\vert\vert \nabla\hat{u} \vert\vert}_2{\vert\vert \hat{b} \vert\vert}_2
+{\vert\vert \nabla\hat{u} \vert\vert}_2{\vert\vert \hat{b} \vert\vert}_2^\frac{1}{2}{\vert\vert \nabla\hat{b} \vert\vert}_2^\frac{1}{2}\right)\\
& \leq\left(\frac{4C^2k_0^2}{\mu}+\frac{64C^4k_0^4}{\mu^3}\right){\vert\vert \hat{b} \vert\vert}_2^2
+\frac{\mu}{8}{\vert\vert \nabla\hat{u} \vert\vert}_2^2
+\frac{\mu}{16}{\vert\vert \nabla\hat{b} \vert\vert}_2^2.
\end{align*}
Thus
\begin{align*}
I_{3}\leq& \left(Ck_0+\frac{6C^2k_0^2}{\mu}+\frac{2C^2k_0^2}{\nu}+\frac{16C^4k_0^4}{\mu\nu^2}
+\frac{64C^4k_0^4}{\mu^3}\right){\vert\vert \hat{b} \vert\vert}_2^2
+\left(Ck_0+\frac{2C^2k_0^2}{\mu}\right){\vert\vert \hat{u} \vert\vert}_2^2\\
&+\frac{\nu}{4}{\vert\vert \partial_z\hat{u} \vert\vert}_2^2
+\frac{\mu}{4}{\vert\vert \nabla\hat{u} \vert\vert}_2^2
+\frac{\mu}{4}{\vert\vert \nabla\hat{b} \vert\vert}_2^2\\
=&: \Re_2{\vert\vert \hat{b} \vert\vert}_2^2
+\Re_3{\vert\vert \hat{u} \vert\vert}_2^2
+\frac{\nu}{4}{\vert\vert \partial_z\hat{u} \vert\vert}_2^2
+\frac{\mu}{4}{\vert\vert \nabla\hat{u} \vert\vert}_2^2
+\frac{\mu}{4}{\vert\vert \nabla\hat{b} \vert\vert}_2^2.
\end{align*}
Using Lemma \ref{lem2.1} and the Young inequality, we infer that
$$
\begin{aligned}
I_{4} =& \int_\Omega {\left[(\hat{b}\cdot\nabla)u+\hat{b}_3\partial_zu\right] \hat{b}} \,{\rm d}\Omega\\
\leq& \int_M {\left(\int_{-1}^1 {(\vert \hat{b} \vert+\vert \partial_z \hat{b} \vert)} \,{\rm d}z\right)
\left(\int_{-1}^1 {\vert \nabla u \vert \vert \hat{b} \vert} \,{\rm d}z\right)} \,{\rm d}M\\
&+ \int_M {\left(\int_{-1}^1 {\nabla \hat{b}} \,{\rm d}z\right)\left(\int_{-1}^1 {\vert \partial_z u \vert \vert \hat{b} \vert} \,{\rm d}z\right)} \,{\rm d}M\\
\leq& C{\vert\vert \hat{b} \vert\vert}_2^{\frac{1}{2}}\left({\vert\vert  \hat{b} \vert\vert}_2^{\frac{1}{2}}+{\vert\vert \nabla \hat{b} \vert\vert}_2^{\frac{1}{2}}\right)
{\vert\vert \nabla u \vert\vert}_2
{\vert\vert \hat{b} \vert\vert}_2^{\frac{1}{2}}\left({\vert\vert \hat{b} \vert\vert}_2^{\frac{1}{2}}+{\vert\vert \nabla \hat{b} \vert\vert}_2^{\frac{1}{2}}\right)\\
&+C{\vert\vert \partial_z\hat{b} \vert\vert}_2
{\vert\vert \nabla u \vert\vert}_2^{\frac{1}{2}}\left({\vert\vert  \nabla u \vert\vert}_2^{\frac{1}{2}}+{\vert\vert \Delta u \vert\vert}_2^{\frac{1}{2}}\right)
\left({\vert\vert \hat{b} \vert\vert}_2+{\vert\vert  \hat{b} \vert\vert}_2^{\frac{1}{2}}{\vert\vert \nabla \hat{b} \vert\vert}_2^{\frac{1}{2}}\right)\\
&+C{\vert\vert \nabla\hat{b} \vert\vert}_2
{\vert\vert \partial_z u \vert\vert}_2^{\frac{1}{2}}\left({\vert\vert  \partial_z u \vert\vert}_2^{\frac{1}{2}}+{\vert\vert \nabla\partial_z u \vert\vert}_2^{\frac{1}{2}}\right)
\left({\vert\vert \hat{b} \vert\vert}_2+{\vert\vert  \hat{b} \vert\vert}_2^{\frac{1}{2}}{\vert\vert \nabla \hat{b} \vert\vert}_2^{\frac{1}{2}}\right)\\
=&:I_{41}+I_{42}+I_{43}.
\end{aligned}
$$
Using the Young inequality, we infer that
\begin{align*}
I_{41} & \leq Ck_0{\vert\vert \hat{b} \vert\vert}_2\left({\vert\vert \hat{b} \vert\vert}_2+{\vert\vert \nabla\hat{b} \vert\vert}_2+{\vert\vert \hat{b} \vert\vert}_2^{\frac{1}{2}}{\vert\vert \nabla\hat{b} \vert\vert}_2^{\frac{1}{2}}\right)\\
& \leq Ck_0{\vert\vert \hat{b} \vert\vert}_2^2+Ck_0{\vert\vert \hat{b} \vert\vert}_2{\vert\vert \nabla\hat{b} \vert\vert}_2+2Ck_0{\vert\vert \hat{b} \vert\vert}_2^{\frac{1}{2}}{\vert\vert \nabla\hat{b} \vert\vert}_2^{\frac{1}{2}}\\
& \leq \left[Ck_0+\frac{4C^2k_0^2}{\mu}+\frac{3}{4}(2Ck_0)^\frac{4}{3}\left(\frac{4}{\mu}\right)^\frac{1}{3}\right]
{\vert\vert \hat{b} \vert\vert}_2^2+\frac{\mu}{8}{\vert\vert \nabla\hat{b} \vert\vert}_2^2,\\
I_{42} & \leq Ck_0{\vert\vert \partial_z\hat{b} \vert\vert}_2\left({\vert\vert \hat{u} \vert\vert}_2
+{\vert\vert  \hat{b} \vert\vert}_2^{\frac{1}{2}}{\vert\vert \nabla \hat{b} \vert\vert}_2^{\frac{1}{2}}\right)\\
& \leq \frac{\nu}{4}{\vert\vert \partial_z\hat{b} \vert\vert}_2^2
+\frac{2C^2k_0^2}{\nu}\left({\vert\vert \hat{b} \vert\vert}_2
+{\vert\vert  \hat{b} \vert\vert}_2{\vert\vert \nabla \hat{b} \vert\vert}_2\right)\\
& \leq \left(\frac{2C^2k_0^2}{\nu}+\frac{16C^4k_0^4}{\mu\nu^2}\right){\vert\vert \hat{b} \vert\vert}_2^2
+\frac{\nu}{4}{\vert\vert \partial_z\hat{b} \vert\vert}_2^2+\frac{\mu}{16}{\vert\vert \nabla\hat{b} \vert\vert}_2^2,\\
I_{43} & \leq Ck_0{\vert\vert \nabla \hat{b} \vert\vert}_2{\vert\vert \hat{b} \vert\vert}_2
+Ck_0{\vert\vert \nabla \hat{b} \vert\vert}_2^\frac{3}{2}{\vert\vert \hat{b} \vert\vert}_2^\frac{1}{2}\\
& \leq \frac{\mu}{8}{\vert\vert \nabla \hat{b} \vert\vert}_2^2+\left(\frac{4C^2k_0^2}{\mu}
+\frac{432C^4k_0^4}{\mu^3}\right){\vert\vert \hat{b} \vert\vert}_2^2.
\end{align*}
Thus
\begin{align*}
I_{4}\leq& \frac{5\mu}{16}{\vert\vert \nabla \hat{b} \vert\vert}_2^2
+\frac{\nu}{4}{\vert\vert \partial_z\hat{b} \vert\vert}_2^2\\
&+\left[Ck_0+\frac{8C^2k_0^2}{\mu}+\frac{3}{4}(2Ck_0)^\frac{4}{3}\left(\frac{4}{\mu}\right)^\frac{1}{3}
+\frac{2C^2k_0^2}{\nu}+\frac{16C^4k_0^4}{\mu\nu^2}
+\frac{432C^4k_0^4}{\mu^3}\right]{\vert\vert \hat{b} \vert\vert}_2^2\\
=&: \frac{5\mu}{16}{\vert\vert \nabla \hat{b} \vert\vert}_2^2
+\frac{\nu}{4}{\vert\vert \partial_z\hat{b} \vert\vert}_2^2+\Re_1{\vert\vert \hat{b} \vert\vert}_2^2.
\end{align*}
We can also get
\begin{align*}
J_{1}&=\int_\Omega {\nabla \hat{p}\cdot \hat{u}} \,{\rm d}\Omega=0,\\
J_{2}&=-\beta_u\int_\Omega {I_h(\hat{u}) \hat{u}} \,{\rm d}\Omega\\
&=-\beta_u\int_\Omega {\left(I_h(\hat{u})-\hat{u}\right) \hat{u}} \,{\rm d}\Omega-\beta_u\int_\Omega {\hat{u} \hat{u}} \,{\rm d}\Omega\\
&\leq \beta_u {\vert\vert I_h(\hat{u})-\hat{u} \vert\vert}_2{\vert\vert \hat{u} \vert\vert}_2-\beta_u {\vert\vert \hat{u} \vert\vert}_2^2\\
&\leq C_h^1h\beta_u \left({\vert\vert \nabla\hat{u} \vert\vert}_2+{\vert\vert \partial_z\hat{u} \vert\vert}_2\right){\vert\vert \hat{u} \vert\vert}_2-\beta_u {\vert\vert \hat{u} \vert\vert}_2^2\\
&\leq \left(\frac{4{(C_h^1)}^2h^2\beta_u^2}{\mu}+\frac{{(C_h^1)}^2h^2\beta_u^2}{\nu}-\beta_u\right){\vert\vert \hat{u} \vert\vert}_2^2
+\frac{\mu}{16}{\vert\vert \nabla\hat{u} \vert\vert}_2^2+\frac{\nu}{4}{\vert\vert \partial_z\hat{u} \vert\vert}_2^2,\\
J_{3}&=-\beta_b\int_\Omega {I_h(\hat{b}) \hat{b}} \,{\rm d}\Omega\\
&\leq \left(\frac{4{(C_h^1)}^2h^2\beta_b^2}{\mu}+\frac{{(C_h^1)}^2h^2\beta_b^2}{\nu}-\beta_b\right){\vert\vert \hat{b} \vert\vert}_2^2
+\frac{\mu}{16}{\vert\vert \nabla\hat{b} \vert\vert}_2^2+\frac{\nu}{4}{\vert\vert \partial_z\hat{b} \vert\vert}_2^2.
\end{align*}
Thus
\begin{align*}
&\frac{1}{2}\frac{d}{dt}\left({\vert\vert \hat{u} \vert\vert}_2^2+{\vert\vert \hat{b} \vert\vert}_2^2\right)
+\mu{\vert\vert \nabla{\hat{u}} \vert\vert}_2^2+\nu{\vert\vert \partial_z{\hat{u}} \vert\vert}_2^2
+\kappa{\vert\vert \nabla{\hat{b}} \vert\vert}_2^2+\sigma{\vert\vert \partial_z{\hat{b}} \vert\vert}_2^2\\
\leq& (\Re_1+\Re_2+\Re3)\left({\vert\vert \hat{u} \vert\vert}_2^2+{\vert\vert \hat{b} \vert\vert}_2^2\right)+\frac{3\nu}{4}\left({\vert\vert \partial_z\hat{u} \vert\vert}_2^2+{\vert\vert \partial_z\hat{b} \vert\vert}_2^2\right)+\frac{7\mu}{8}\left({\vert\vert \nabla\hat{u} \vert\vert}_2^2+{\vert\vert \nabla\hat{b} \vert\vert}_2^2\right)\\
&+\left(\frac{4{(C_h^1)}^2h^2\beta_u^2}{\mu}+\frac{{(C_h^1)}^2h^2\beta_u^2}{\nu}-\beta_u\right){\vert\vert \hat{u} \vert\vert}_2^2
+\left(\frac{4{(C_h^1)}^2h^2\beta_b^2}{\mu}+\frac{{(C_h^1)}^2h^2\beta_b^2}{\nu}-\beta_b\right){\vert\vert \hat{b} \vert\vert}_2^2,
\end{align*}
where
\begin{align}
&\Re_1=Ck_0+\frac{8C^2k_0^2}{\mu}+\frac{3}{4}(2Ck_0)^\frac{4}{3}\left(\frac{4}{\mu}\right)^\frac{1}{3}+
\frac{2C^2k_0^2}{\nu}+\frac{16C^4k_0^4}{\mu\nu^2}+\frac{432C^4k_0^4}{\mu^3},\label{r1}\\
&\Re_2=Ck_0+\frac{6C^2k_0^2}{\mu}+\frac{2C^2k_0^2}{\nu}+\frac{16C^4k_0^4}{\mu\nu^2}
+\frac{64C^4k_0^4}{\mu^3},\label{r2}\\
&\Re_3=Ck_0+\frac{2C^2k_0^2}{\mu}.\label{r3}
\end{align}
So if $\beta_u,\beta_b \; {\rm{and}}\; h$ satisfy
$$
\beta_u \geq 2(\Re_1+\Re_2+\Re_3),\;
\beta_b \geq 2(\Re_1+\Re_2+\Re_3),
$$
and
$$
0 < h \leq \min\left\{\frac{1}{\beta_u C_h^1}\sqrt{\frac{\mu\nu}{\mu+4\nu}},\frac{1}{\beta_b C_h^1}\sqrt{\frac{\mu\nu}{\mu+4\nu}}\right\},
$$
then the Gronwall inequality implies that $(\tilde{u},\tilde{b})$ converges to $(u,b)$ exponentially in $L^2$.
\end{proof}

\section{Sensitivity analysis}

In this section, we consider the sensitivity analysis problem for the assimilation system \eqref{2.11}-\eqref{2.14}, by proving that a sequence of solutions of difference quotients equations converges to a unique solution of the formal sensitivity equations. We only analyze the horizontal viscosity and magnetic diffusivity, and the vertical direction is similar. To simplify the notation, we set $\kappa=\mu,\sigma=\nu$. We first present the sensitivity equations and the difference quotients equations, respectively.

By taking the derivative of the assimilation system \eqref{2.11}-\eqref{2.14} with respect to $\mu$, we can get the sensitivity equations with respect to the horizontal direction:
\begin{align}
&\partial_t{\check{u}}+B(\check{u},\tilde{u})+B(\tilde{u},\check{u})+\nabla \check{p}
-B(\check{b},\tilde{b})-B(\tilde{b},\check{b})-\Delta \tilde{u}-\mu\Delta\check{u}-\nu{\partial_z^2}\check{u}=-\beta_u {I_h(\check{u})},\label{4.1}\\
&\partial_z{\check{p}}=0,\\
&\partial_t{\check{b}}+B(\check{u},\tilde{b})+B(\tilde{u},\check{b})
-B(\tilde{b},\check{u})-B(\check{b},\tilde{u})-\Delta \tilde{b}-\mu\Delta \check{b}-\nu{\partial_z^2}\check{b}=-\beta_b {I_h(\check{b})},\\
&\nabla\cdot{\check{u}}+\partial_z{\check{u}}_3=0,\; \nabla\cdot{\check{b}}+\partial_z{\check{b}}_3=0,\label{4.4}
\end{align}
where
$\check{u}=\partial_\mu{\tilde{u}}, \check{b}=\partial_\mu{\tilde{b}}, \check{p}=\partial_\mu{\tilde{p}},
\check{u}_3=\partial_\mu{\tilde{u}_3},
\check{b}_3=\partial_\mu{\tilde{b}_3}.
$

The difference quotients equations is obtained as the following by taking the solutions $(\tilde{u}_1,\tilde{b}_1)$ and $(\tilde{u}_2,\tilde{b}_2)$ of the assimilation system $\mu_1$ and $\mu_2$ respectively:
\begin{align}
&\partial_t{\bar{u}}+B(\bar{u},\tilde{u}_1)+B(\tilde{u}_2,\bar{u})+\nabla \bar{p}-B(\bar{b},\tilde{b}_1)-B(\tilde{b}_2,\bar{b})-\Delta\tilde{u}_1-\mu_2\Delta\bar{u}-\nu{\partial_z^2}\bar{u}=-\beta_u {I_h(\bar{u})},\label{4.5}\\
&\partial_z {\bar{p}}=0, \\
&\partial_t {\bar{b}}+B(\bar{u},\tilde{b}_1)+B(\tilde{b}_2,\bar{u})
-B(\bar{b},\tilde{u}_1)-B(\tilde{b}_2,\bar{u})-\Delta\tilde{b}_1-\mu_2\Delta\bar{b}-\nu{\partial_z^2}\bar{b}=-\beta_b {I_h(\bar{b})},\\
&\nabla \cdot{\bar{u}}+\partial_z{\bar{u}}_3=0,\; \nabla\cdot{\bar{b}}+\partial_z{\bar{b}}_3=0,\label{4.8}
\end{align}
where
$\bar{u}=\frac{\tilde{u}_1-\tilde{u}_2}{\mu_1-\mu_2}, \bar{b}=\frac{\tilde{b}_1-\tilde{b}_2}{\mu_1-\mu_2}, \bar{p}=\frac{\tilde{p}_1-\tilde{p}_2}{\mu_1-\mu_2}$.

Next, we consider the well-posedness of the sensitivity equations and the difference quotients equations, respectively.
\begin{proposition}\label{prop.dqs}
Suppose that $(\tilde{u}_1,\tilde{b}_1)$ and $(\tilde{u}_2,\tilde{b}_2)$ are two strong solutions of the assimilation system \eqref{2.11}-\eqref{2.14} with viscosities $\mu_1$ and $\mu_2$, respectively. Then, the difference quotients system \eqref{4.5}-\eqref{4.8} has a unique solution $(\bar{u},\bar{b})$.
\end{proposition}

\begin{proof} It is estimated from uniform boundedness that the difference quotients equation has a strong $H^1$ solution. Next, we prove the unique.

Let  $(\bar{u}_1,\bar{b}_1)$ and $(\bar{u}_2,\bar{b}_2)$ be two solutions of the difference quotients system,
and $\dot{u}=\bar{u}_1-\bar{u}_2, \dot{b}=\bar{b}_1-\bar{b}_2$.
We have
\begin{align*}
&\partial_t{\dot{u}}+B(\dot{u},\tilde{u}_1)+B(\tilde{u}_2,\dot{u})+\nabla \dot{p}-B(\dot{b},\tilde{b}_1)-B(\tilde{b}_2,\dot{b})-\mu_2\Delta\dot{u}-\nu{\partial_z^2}\dot{u}=-\beta_u {I_h(\dot{u})},\\
&\partial_z{\bar{p}}=0,\\
&\partial_t{\dot{b}}+B(\dot{u},\tilde{b}_1)+B(\tilde{b}_2,\dot{u})-B(\dot{b},\tilde{u}_1)-B(\tilde{b}_2,\dot{u})-\mu_2\Delta\dot{b}-\nu{\partial_z^2}\dot{b}=-\beta_b {I_h(\dot{b})},\\
&\nabla\cdot{\dot{u}}+\partial_z{\dot{u}}_3=0,\; \nabla\cdot{\dot{b}}+\partial_z{\dot{b}}_3=0.
\end{align*}
Taking the inner product of equation with $-\Delta\dot{u}-\partial_z^2 \dot{u}$ and $-\Delta\dot{b}-\partial_z^2 \dot{b}$ respectively, we can get
$$
\begin{aligned}
&\frac{1}{2}\frac{d}{dt}\left({\vert\vert \nabla\dot{u} \vert\vert}_2^2+{\vert\vert \partial_z\dot{u} \vert\vert}_2^2\right)
+\frac{1}{2}\frac{d}{dt}\left({\vert\vert \nabla\dot{b} \vert\vert}_2^2+{\vert\vert \partial_z\dot{b} \vert\vert}_2^2\right)\\
&+\int_\Omega {\left(\mu_2\Delta\dot{u}+\nu{\partial_z^2}\dot{u}\right)(\Delta\dot{u}+{\partial_z^2}\dot{u})} \,{\rm d}\Omega
+\int_\Omega {(\mu_2\Delta\dot{b}+\nu{\partial_z^2}\dot{b})(\Delta\dot{b}+{\partial_z^2}\dot{b})} \,{\rm d}\Omega\\
=&\int_\Omega {\nabla \dot{p}\cdot (\Delta\dot{u}+{\partial_z^2}\dot{u})} \,{\rm d}\Omega\\
&+\int_\Omega {\left(B(\dot{u},\tilde{u}_1)+B(\tilde{u}_2,\dot{u})\right) (\Delta\dot{u}+{\partial_z^2}\dot{u})} \,{\rm d}\Omega
-\int_\Omega {\left(B(\dot{b},\tilde{b}_1)+B(\tilde{b}_2,\dot{b})\right) (\Delta\dot{u}+{\partial_z^2}\dot{u})} \,{\rm d}\Omega\\
&+\int_\Omega {\left(B(\dot{u},\tilde{b}_1)+B(\tilde{b}_2,\dot{u})\right) (\Delta\dot{b}+{\partial_z^2}\dot{b})} \,{\rm d}\Omega
-\int_\Omega {\left(B(\dot{b},\tilde{u}_1)+B(\tilde{b}_2,\dot{b})\right) (\Delta\dot{b}+{\partial_z^2}\dot{b})} \,{\rm d}\Omega\\
&+\beta_u\int_\Omega {I_h(\dot{u}) (\Delta\dot{u}+{\partial_z^2}\dot{u})} \,{\rm d}\Omega
+\beta_b\int_\Omega {I_h(\dot{b}) (\Delta\dot{b}+{\partial_z^2}\dot{b})} \,{\rm d}\Omega\\
=&:K_{1}+W_{1}+W_{2}+W_{3}+W_{4}+W_{5}+W_{6}.
\end{aligned}
$$
Using the H\"older and Young inequality, we infer that
$$
\begin{aligned}
\int_\Omega {\left(\mu_2\Delta\dot{u}+\nu{\partial_z^2}\dot{u}\right)(\Delta\dot{u}+{\partial_z^2}\dot{u})} \,{\rm d}\Omega &\geq \frac{15}{16}\min\{\mu_2, \nu\}{\vert\vert \dot{u} \vert\vert}_{H^2}^2-C{\vert\vert \dot{u} \vert\vert}_{H^1}^2,\\
\int_\Omega {(\mu_2\Delta\dot{b}+\nu{\partial_z^2}\dot{b})(\Delta\dot{b}+{\partial_z^2}\dot{b})} \,{\rm d}\Omega
&\geq \frac{15}{16}\min\{\mu_2, \sigma\}{\vert\vert \dot{b} \vert\vert}_{H^2}^2-C{\vert\vert \dot{b} \vert\vert}_{H^1}^2.\\
\end{aligned}
$$
Using Lemma \ref{lem2.3} and Young's inequality, we infer that
$$
\begin{aligned}
W_{1}&\leq C{\vert\vert \dot{u} \vert\vert}_{H^1}^{\frac{1}{2}}{\vert\vert \dot{u} \vert\vert}_{H^2}^{\frac{1}{2}}
{\vert\vert \tilde{u}_1 \vert\vert}_{H^1}^{\frac{1}{2}}{\vert\vert \tilde{u}_1 \vert\vert}_{H^2}^{\frac{1}{2}}
{\vert\vert \dot{u} \vert\vert}_{H^2}+C{\vert\vert \tilde{u}_2 \vert\vert}_{H^1}^{\frac{1}{2}}{\vert\vert \tilde{u}_2 \vert\vert}_{H^2}^{\frac{1}{2}}
{\vert\vert \dot{u} \vert\vert}_{H^1}^{\frac{1}{2}}{\vert\vert \dot{u} \vert\vert}_{H^2}^{\frac{1}{2}}
{\vert\vert \dot{u} \vert\vert}_{H^2}\\
&\leq C\left({\vert\vert \tilde{u}_1 \vert\vert}_{H^1}^2{\vert\vert \tilde{u}_1 \vert\vert}_{H^2}^2+{\vert\vert \tilde{u}_2 \vert\vert}_{H^1}^2{\vert\vert \tilde{u}_2 \vert\vert}_{H^2}^2\right)\left({\vert\vert \dot{u} \vert\vert}_{H^1}^2+{\vert\vert \dot{u} \vert\vert}_{H^2}^2\right),\\
W_{2}&\leq C{\vert\vert \dot{b} \vert\vert}_{H^1}^{\frac{1}{2}}{\vert\vert \dot{b} \vert\vert}_{H^2}^{\frac{1}{2}}
{\vert\vert \tilde{u}_1 \vert\vert}_{H^1}^{\frac{1}{2}}{\vert\vert \tilde{u}_1 \vert\vert}_{H^2}^{\frac{1}{2}}
{\vert\vert \dot{u} \vert\vert}_{H^2}+C{\vert\vert \tilde{u}_2 \vert\vert}_{H^1}^{\frac{1}{2}}{\vert\vert \tilde{u}_2 \vert\vert}_{H^2}^{\frac{1}{2}}
{\vert\vert \dot{b} \vert\vert}_{H^1}^{\frac{1}{2}}{\vert\vert \dot{b} \vert\vert}_{H^2}^{\frac{1}{2}}
{\vert\vert \dot{u} \vert\vert}_{H^2}\\
&\leq C\left({\vert\vert \tilde{u}_1 \vert\vert}_{H^1}^2{\vert\vert \tilde{u}_1 \vert\vert}_{H^2}^2+{\vert\vert \tilde{u}_2 \vert\vert}_{H^1}^2{\vert\vert \tilde{u}_2 \vert\vert}_{H^2}^2\right)\left({\vert\vert \dot{b} \vert\vert}_{H^1}^2+{\vert\vert \dot{b} \vert\vert}_{H^2}^2+{\vert\vert \dot{u} \vert\vert}_{H^2}^2\right),\\
W_{3}&\leq C{\vert\vert \dot{u} \vert\vert}_{H^1}^{\frac{1}{2}}{\vert\vert \dot{u} \vert\vert}_{H^2}^{\frac{1}{2}}
{\vert\vert \tilde{b}_1 \vert\vert}_{H^1}^{\frac{1}{2}}{\vert\vert \tilde{b}_1 \vert\vert}_{H^2}^{\frac{1}{2}}
{\vert\vert \dot{b} \vert\vert}_{H^2}+C{\vert\vert \tilde{b}_2 \vert\vert}_{H^1}^{\frac{1}{2}}{\vert\vert \tilde{b}_2 \vert\vert}_{H^2}^{\frac{1}{2}}
{\vert\vert \dot{b} \vert\vert}_{H^1}^{\frac{1}{2}}{\vert\vert \dot{b} \vert\vert}_{H^2}^{\frac{1}{2}}
{\vert\vert \dot{b} \vert\vert}_{H^2}\\
&\leq C\left({\vert\vert \tilde{b}_1 \vert\vert}_{H^1}^2{\vert\vert \tilde{b}_1 \vert\vert}_{H^2}^2+{\vert\vert \tilde{b}_2 \vert\vert}_{H^1}^2{\vert\vert \tilde{b}_2 \vert\vert}_{H^2}^2\right)\left({\vert\vert \dot{b} \vert\vert}_{H^1}^2+{\vert\vert \dot{b} \vert\vert}_{H^2}^2+{\vert\vert \dot{u} \vert\vert}_{H^1}^2+{\vert\vert \dot{u} \vert\vert}_{H^2}^2\right),\\
W_{4}&\leq C{\vert\vert \dot{b} \vert\vert}_{H^1}^{\frac{1}{2}}{\vert\vert \dot{b} \vert\vert}_{H^2}^{\frac{1}{2}}
{\vert\vert \tilde{u}_1 \vert\vert}_{H^1}^{\frac{1}{2}}{\vert\vert \tilde{u}_1 \vert\vert}_{H^2}^{\frac{1}{2}}
{\vert\vert \dot{b} \vert\vert}_{H^2}+C{\vert\vert \tilde{u}_2 \vert\vert}_{H^1}^{\frac{1}{2}}{\vert\vert \tilde{u}_2 \vert\vert}_{H^2}^{\frac{1}{2}}
{\vert\vert \dot{b} \vert\vert}_{H^1}^{\frac{1}{2}}{\vert\vert \dot{b} \vert\vert}_{H^2}^{\frac{1}{2}}
{\vert\vert \dot{b} \vert\vert}_{H^2}\\
&\leq C\left({\vert\vert \tilde{u}_1 \vert\vert}_{H^1}^2{\vert\vert \tilde{u}_1 \vert\vert}_{H^2}^2+{\vert\vert \tilde{u}_2 \vert\vert}_{H^1}^2{\vert\vert \tilde{u}_2 \vert\vert}_{H^2}^2\right)\left({\vert\vert \dot{b} \vert\vert}_{H^1}^2+{\vert\vert \dot{b} \vert\vert}_{H^2}^2\right),\\
\end{aligned}
$$
and
$$
\begin{aligned}
K_{1}&= 0,\\
W_{5}&\leq \beta_u{\vert\vert I_h(\dot{u}) \vert\vert}_2{\vert\vert \dot{u} \vert\vert}_{H^2}
\leq \frac{1}{2}\beta_u\left({\vert\vert I_h(\dot{u}) \vert\vert}_2^2+{\vert\vert \dot{u} \vert\vert}_{H^2}^2\right),\\
W_{6}&\leq \beta_b{\vert\vert I_h(\dot{b}) \vert\vert}_2{\vert\vert \dot{b} \vert\vert}_{H^2}
\leq \frac{1}{2}\beta_b\left({\vert\vert I_h(\dot{b}) \vert\vert}_2^2+{\vert\vert \dot{b} \vert\vert}_{H^2}^2\right).\\
\end{aligned}
$$
Combining all the above estimates, we obtain
$$
\begin{aligned}
&\frac{1}{2}\frac{d}{dt}\left({\vert\vert \dot{u} \vert\vert}_{H_1}^2+{\vert\vert \dot{b} \vert\vert}_{H_1}^2\right)
+ \frac{1}{2}\mu^*\left({\vert\vert \dot{u} \vert\vert}_{H_2}^2+{\vert\vert \dot{b} \vert\vert}_{H_2}^2\right)\\
\leq& C\left({\vert\vert \tilde{u}_1 \vert\vert}_{H^1}^2{\vert\vert \tilde{u}_1 \vert\vert}_{H^2}^2+{\vert\vert \tilde{u}_2 \vert\vert}_{H^1}^2{\vert\vert \tilde{u}_2 \vert\vert}_{H^2}^2
+{\vert\vert \tilde{b}_1 \vert\vert}_{H^1}^2{\vert\vert \tilde{b}_1 \vert\vert}_{H^2}^2+{\vert\vert \tilde{b}_2 \vert\vert}_{H^1}^2{\vert\vert \tilde{b}_2 \vert\vert}_{H^2}^2\right)
\left({\vert\vert \dot{u} \vert\vert}_{H_1}^2+{\vert\vert \dot{b} \vert\vert}_{H_1}^2\right),\\
\end{aligned}
$$
then by Gronwall's inequality, the difference quotients system has a unique solution.
\end{proof}

\begin{theorem}
Let $T>0, (\tilde{u}, \tilde{b})$ and $(\tilde{u}_n, \tilde{b}_n)$ be $H^1$ solutions of the assimilation system \eqref{2.11}-\eqref{2.14} in $\mu$ and $\mu_n$, and $(\bar{u}_n, \bar{b}_n)$  be $H^1$ strong solutions of the difference quotients equation \eqref{4.5}-\eqref{4.8} with viscosity $\mu_n$. If $\mu_n\rightarrow\mu,n\rightarrow\infty$, then there is a subsequence of $(\bar{u}_n, \bar{b}_n)$ in $L^2$ that converges to the unique strong solution $(\check{u},\check{b})$ of the sensitive system \eqref{4.1}-\eqref{4.4}.
\end{theorem}
\begin{proof}
Let $T>0$ be given. Let $n>N$ be large enough, which implies $\mu_n\in\left(\frac{1}{2}\mu,\frac{3}{2}\mu\right)$.

Considering these estimates for $H^1$ strong solutions $(\tilde{u}_n, \tilde{b}_n)$, we can easy to verify that
\begin{align*}
{\vert\vert (\tilde{u}_n, \tilde{b}_n) \vert\vert}_{H_1}^2&\leq f(\mu_n) \leq f\left(\frac{\mu}{2}\right),\\
\int_{0}^T {{\vert\vert (\tilde{u}_n, \tilde{b}_n) \vert\vert}_{H_2}^2} \,{\rm d}s&\leq \frac{f(\mu_n)}{\min \{\mu_n,\nu\}}\leq \frac{f(\frac{1}{2}\mu)}{\min \{\frac{1}{2}\mu,\nu\}},
\end{align*}
where $f(\mu_n)$ is a decreasing function respect to $\mu_n$.
So we can get that $\{(\tilde{u}_n, \tilde{b}_n)\}_{n=1}^{\infty}$ is uniformly bounded in $L^2(0,T;H^2)$ respect to $n$.
Then there exists a subsequence such that $(\tilde{u}_n, \tilde{b}_n)\rightarrow(\tilde{u}', \tilde{b}')$.
Similarly available $(\frac{d\tilde{u}_n}{dt}, \frac{d\tilde{b}_n}{dt})$ is also uniformly bounded in $L^2(0,T;L^2)$ respect to $n$.
Therefore, there exists a subsequence such that
\begin{align*}
&\left(\frac{d\tilde{u}_n}{dt}, \frac{d\tilde{b}_n}{dt}\right)\rightharpoonup\left(\frac{d\tilde{u}'}{dt}, \frac{d\tilde{b}'}{dt}\right) \;{\rm{in}}\; L^2(0,T;L^2),\\
&\left(\tilde{A}_u\tilde{u}_n, \tilde{A}_b\tilde{b}_n\right)\rightharpoonup\left(\tilde{A}_u\tilde{u}', \tilde{A}_b\tilde{b}'\right) \;{\rm{in}}\; L^2(0,T;L^2),\\
& \left(B(\tilde{u}_n, \tilde{u}_n),B(\tilde{b}_n, \tilde{b}_n)\right)\rightharpoonup\left(B(\tilde{u}', \tilde{u}'),B(\tilde{b}', \tilde{b}')\right) \;{\rm{in}}\; L^2(0,T;L^2),\\
& \left(B(\tilde{u}_n, \tilde{b}_n),B(\tilde{b}_n, \tilde{u}_n)\right)\rightharpoonup\left(B(\tilde{u}', \tilde{b}'),B(\tilde{b}', \tilde{u}')\right) \;{\rm{in}}\; L^2(0,T;L^2).
 \end{align*}
So $(\tilde{u}', \tilde{b}')$ is a $H^1$ strong solution to assilation systems.
Due to the uniqueness of solutions, we know that $(\tilde{u}', \tilde{b}')=(\tilde{u}, \tilde{b})$.

For difference quotients equation, we take $(\tilde{u}_1, \tilde{b}_1)=(\tilde{u}, \tilde{b})$, $(\tilde{u}_2, \tilde{b}_2)=(\tilde{u}_n, \tilde{b}_n)$.
$(\bar{u}_n, \bar{b}_n)$ is the corresponding strong solution to the difference quotients equation with viscosity $\mu_n$.

Similar to the estimate of Proposition {\ref{prop.dqs}}, we can get that $(\bar{u}_n, \bar{b}_n)$ is uniformly bounded in $H^1$ respect to $n$.
Then there exists a subsequence such that
\begin{align*}
&(\bar{u}_n, \bar{b}_n)\stackrel{*}{\rightharpoonup}(\check{u},\check{b}) \;{\rm{in}}\; L^\infty(0,T;H^1),\\
&(\bar{u}_n, \bar{b}_n)\rightharpoonup(\check{u},\check{b}) \;{\rm{in}}\; L^2(0,T;H^2).
\end{align*}

Similarly, $(\frac{d\check{u}_n}{dt}, \frac{d\check{b}_n}{dt})$ is also uniformly bounded in $L^2(0,T;L^2)$ respect to $n$.
Combining the Aubin Lemma, we deduce that
$(\bar{u}_n, \bar{b}_n) \rightarrow (\check{u},\check{b})$ in $L^2(0,T;H^1)$. By Lemma \ref{lem2.4}, we get $(\check{u},\check{b})\in C(0,T;H^1)$.
\end{proof}

\section*{Acknowledgments}

This work was supported by the National Natural Science Foundation  of China (No. 12271261), the Qing Lan Project of Jiangsu Province and
Postgraduate Research and Practice Innovation Program of Jiangsu Province (Grant No. KYCX22\_1125).

\section*{Conflict of Interest}
This work does not have any conflicts of interest.

\end{document}